\documentclass[showkeys]{amsart}
\usepackage{amsfonts}
\usepackage{amsthm}
\usepackage{amsmath}
\usepackage{geometry}
\usepackage{tikz}
\usepackage{tikz-cd}
\usetikzlibrary{matrix, calc, arrows}

\usepackage[varg]{txfonts}
\usepackage{graphicx}
\usepackage{bm}
\usepackage{enumitem}
\usepackage{indentfirst}
\usepackage{comment}
\usepackage{ulem}
\usepackage{color}
\usepackage{cases}

\allowdisplaybreaks

\allowdisplaybreaks

\newtheorem{thm}{Theorem}[section]
\newtheorem{rem}[thm]{Remark}
\newtheorem{ex}[thm]{Example}

\newtheorem{lem}[thm]{Lemma}
\newtheorem{prop}[thm]{Proposition}

\theoremstyle{definition}

\renewcommand{\Bbb}{\mathbb}
\def\NZQ{\Bbb}

\def\NN{{\NZQ N}}

\def\QQ{{\NZQ Q}}

\def\ZZ{{\NZQ Z}}

\def\ker{{\text{Ker }}}
\def\ima{{\text{Im }}}
\def\rank{{\text{rank }}}

\def\M{{\mathcal M}}

\def\bl{\ | \ }
\DeclareMathOperator{\coker}{coker}

\title{A splitting property of the chromatic homology of the complete graph}
\author{So Yamagata}
\address{Department of Applied Mathematics, Faculty of Science, Fukuoka University, Fukuoka, 814-0180, Japan.}
\email{so.yamagata@fukuoka-u.ac.jp}

\begin{document}
\begin{abstract}
    Khovanov \cite{K} introduced a bigraded cohomology theory of links whose graded Euler characteristic is the Jones polynomial. The theory was subsequently applied to the chromatic polynomial of graph \cite{HY}, resulting in a categorification known as the ``chromatic homology''. Much as in the Khovanov homology, the chromatic polynomial can be obtained by taking the Euler characteristic of the chromatic homology. In the present paper, we introduce a combinatorial description of enhanced states that can be applied to analysis of the homology in an explicit way by hand. Using the new combinatorial description, we show a splitting property of the chromatic homology for a certain class of graphs. Finally, as an application of the description, we compute the chromatic homology of the complete graph.
\end{abstract}

\maketitle
\section{Introduction}
Khovanov \cite{K} introduced a bigraded cohomology theory of links whose graded Euler characteristic is the Jones polynomial. The theory was subsequently applied to the chromatic polynomial of graph \cite{HY}, resulting in a categorification known as the ``chromatic homology''. Much as in the Khovanov homology, the chromatic polynomial can be obtained by taking the Euler characteristic of the chromatic homology. Several results on the chromatic homology have been obtained. In 2006, Helme-Guizon et al. \cite{H-GPR} studied torsions in the chromatic homology and presented a vanishing theorem of the homology based on their results. Specifically, they determined which graphs have the homology that contains torsion. They also proved a thickness-type theorem for the homology groups, and gave computations of the homology of polygon graphs with coefficients in the general algebra. A study by Lawrance and Sazdanovic \cite{LS} showed that the torsion of the chromatic homology is of order two. The first group of the homology was studied by Pabiniak et al. \cite{PPS}, and they also gave many interesting conjecture about the homology with algebras other than $\mathcal{A}_2 = \mathbb{Z}/(x^2)$. Helme-Guizon and colleagues \cite{CCR} showed that the chromatic homology with a rational coefficient can be determined by the chromatic polynomial, proving that the homologies of the ``knight'' pair are isomorphic. In 2018, Sazdanovic and Scofield \cite{SS} studied the span of the homology and considered how the homology changes when a cycle graph is added to the given graph. The chromatic homology with arbitral algebra was observed in a study by Helme-Guizon and Rong \cite{HR}. Providing another perspective, homology theories for the chromatic polynomial have also been observed \cite{CS}, \cite{Z}.

The chromatic homology is interesting not only in itself but also in relation to other areas of study. The relation to Hochschild homology was investigated by Przytycki \cite{P}, who showed that the Hochschild homology of the unital algebra is isomorphic to the chromatic homology over the algebra of a cycle graph. With respect to the topology of configuration spaces, Baranovsky and Sazdanovic \cite{BS} showed that the $E_1$-term of the Bendersky-Gitler-type spectral sequence converging to the homology of the graph configuration space is given by the chromatic complex. B\"{o}kstedt and Minuz \cite{BM} subsequently studied the relation between the work of Baranovsky and Sazdanovic \cite{BS} and Kriz's rational model for the configuration space \cite{Kr}.

There are also variants of the chromatic homology. In an analysis by Jasso-Hernandez and Rong \cite{JH-R}, the Tutte homology was provided as a categorification of the Tutte polynomial. The categorification of the chromatic polynomial of embedded graphs was studied by Loebl and Moffatt \cite{LM}. The categorification of the Stanley's chromatic symmetric function was introduced by Sazdanovic and Yip \cite{SY}. As an analogy of the chromatic homology, Dancsco and Licata \cite{DL} provided several homology theories for hyperplane arrangement as a categorification of several polynomials associated with the combinatorics of hyperplane arrangement. In particular, it is easily seen that the characteristic homology, a categorification of the characteristic polynomial, of the braid arrangement is isomorphic to the chromatic homology of the complete graph. 

In the present paper, we introduce a combinatorial description of enhanced states which would be useful to analyze the homology in an explicit way by hand. Using the description we show a splitting property of the chromatic homology for a certain class of graphs. Let $K_{n-1}^m$ be a graph obtained by adding $m$ edges connecting a single vertex with each vertex of the complete graph $K_{n-1}$ and $e$ be one of the connecting edges (see Subsection \ref{subsec:combi_descr}). Then, we show the following theorem.
\begin{thm}[Theorem \ref{thm:main_split}]
    For $n \geq 4$ let $G = K_{n-1}^m$ and $e$ be an edge chosen as above. Then, we have the following split exact sequence
    \begin{equation}
        0 \to H^{i,j}(G/e) \to H^{i+1, j} (G) \to H^{i+1,j} (G - e) \to 0
    \end{equation}
    for all $i,j$ with $i + j = n - 1, n$. \\
    If we sum over $j$, we have the split exact sequence
    \begin{equation}
        0 \to H^{i}(G/e) \to H^{i+1} (G) \to H^{i+1} (G - e) \to 0
    \end{equation}
    for all $i$.
\end{thm}

This result would allow us to compute the chromatic homology in an inductive way for the graph $K_{n-1}^m$. Actually, as an application of the theorem, we can describe the chromatic homology of the complete graph recursively. The description of the homology was firstly conjectured by Hasegawa and the author \cite{H}.
\begin{thm}[Conjecture 6.8 \cite{H}]\label{thm:complete}
    For $n \geq 4$ the chromatic homology groups of a complete graph $K_n$ with $n$ vertices are given as
    \begin{equation}
        H^i (K_n) =
        \begin{cases}
            \ZZ \{ n \}                                                  & i = 0             \\
            H^{i-1} (K_{n-1})^{\oplus (n-2)} \oplus H^i(K_{n-1}) \{ 1 \} & 1 \leq i \leq n-2 \\
            0                                                            & i \geq n-1.
        \end{cases}
    \end{equation}
\end{thm}
Remark that Theorem \ref{thm:complete} also gives the characteristic homology, introduced in \cite{DL}, of the braid arrangement, which would be the first result of the explicit calculation of the homology. 

This paper is organized as follows. In Section \ref{sec:prem}, we recall basic notions from graph theory, and the construction of the chromatic homology. In Section \ref{sec:main}, we introduce the combinatorial description of enhanced states, and show a splitting property of the chromatic homology. In Section \ref{sec:complete}, we compute the chromatic homology of the complete graph.
\section{Preliminaries}\label{sec:prem}
\subsection{Graph and its chromatic polynomial}
In this subsection we review the basic notions of graph theory. \\
Let $G = (V(G), E(G))$ be a graph with vertex set $V(G)$ and edge set $E(G)$. If there is an order on the set $E(G)$, the graph is called \textbf{ordered}. Throughout this paper we assume the following.
\begin{itemize}
    \item The graph $G$ is connected;
    \item The vertices of $G$ are indexed by $\{ i \in \NN \mid 1 \leq i \leq \#V(G) \}$;
    \item The graph $G$ is ordered lexicographically with respect to pairs of numbers representing edges, i.e., for $\{ i_1, i_2 \}$, $\{ j_1, j_2 \} \in E(G)$ with $i_1 < i_2$, $j_1 < j_2$, $\{ i_1, i_2 \} < \{ j_1, j_2 \}$ if $(i_1, i_2) < (j_1, j_2)$ as a lexicographic order.
\end{itemize}
Let us take an edge $e \in E(G)$ of a graph $G$. We define the \textbf{deletion} of $G$ denoted by $G - e$ as a graph obtained by just deleting $e$ from $G$, and the \textbf{contraction} of $G$ denoted by $G/e$ as a graph obtained by collapsing two end vertices of $e$ into a single vertex along $e$. For a subset $s \subset E(G)$, a \textbf{spanning graph} denoted by $[G : s]$ is a graph $(V(G), s)$. An edge $e \in E(G)$ is called a \textbf{bridge} if the number of connected components of $G - e$ is one more than that of $G$. \\
For a positive integer $\lambda$ define a coloring by a map $c: V(G) \to [\lambda]$ with a condition that $c(i) \neq c(j)$, $i,j \in V(G)$ if $\{ i,j \} \in E(G)$. Let $P_G(\lambda)$ be the number of different colorings of a graph $G$ using at most $\lambda$ colors. For any graph $G$ the $P_G(\lambda)$ is a well-defined polynomial of $\lambda$ known as the \textbf{chromatic polynomial}. It is well-known that the chromatic polynomial satisfies the \textbf{deletion-contraction relation}, i.e., for any edge $e \in E(G)$ the relation
\begin{equation}
    P_G(\lambda) = P_{G-e}(\lambda) - P_{G/e}(\lambda)
\end{equation}
holds. 
\subsection{Chromatic homology}\label{subsec:ch_hlg}
In this subsection, we review the construction of the chromatic homology. Most of the exposition here is based on \cite{HY}. 
Let $\displaystyle \M = \oplus_{j \geq 0} M_j$ be a graded $\ZZ$-module, where $\{ M_j \}$ denotes the set of homogeneous elements with degree $j$. We call the power series
\begin{equation*}
    q \dim \M = \sum_{j \geq 0} q^j \cdot \rank (M_j)
\end{equation*}
the graded dimension of $\M$, where $\rank (M_j) = \dim_{\QQ} M_j \otimes_{\ZZ} \QQ$. For a graded $\ZZ$-module $\M$ we define $\M\{ l \}_j = M_{j-l}$; that is, all of the degrees are increased by $l$, and the module satisfies $q \dim \M \{ l \} = q^l \cdot q \dim \M$.  

Helme-Guizon and Rong \cite{HY} give two equivalent constructions of the chromatic homology. One is the cubic complex construction, and the other is the enhanced state construction. For our purpose, it is sufficient to review only the latter construction. \\
Let $G = (V(G), E(G))$ be an ordered graph and $s \subset E(G)$. Let $E_1, \dots, E_d$ be connected components of a spanning graph $[G:s]$. Consider a map $c: \cup_{h=1}^d E_h \to \ZZ[x]/(x^2)$ called the coloring which gives a color 1 or $x$ on each component $E_h$, $h=1, \dots, d$ of the graph $G$. We call the colored graph an \textbf{enhanced state} of $G$ and denote it by $S = (s,c)$.
For an enhanced state $S = (s,c)$ define
\begin{equation*}
    i(S) = \# s, \ \text{and} \ j(S) = \#\{ h \in [d] \ | \ c(E_h) = x \}.
\end{equation*}
Let $C^{i,j}(G)$ be a $\ZZ$-module generated by enhanced states $S$ of $G$ with $i(S) = i$ and $j(S) = j$. We define the differential $\partial_{G}^{i,j}: C^{i,j}(G) \to C^{i+1,j}(G)$ by
\begin{equation}
    \partial_{G}^{i,j}(G) = \sum_{e \in E(G) \setminus s} (-1)^{n(e)} S_e,
\end{equation}
where $n(e)$ is the number of edges in $s$ that are ordered before $e$ and $S_e = (s_e, c_e)$ is an enhanced state defined as follows. Let $s_e = s \cup \{ e \}$ and $E_1, \dots, E_d$ be the components of $[G:s]$. If $e$ is a bridge of $E_a$ and $E_b$, $a \neq b$, then define a map $c_e(E_a \cup E_b \cup \{ e \}) = c(E_a) c(E_b)$. If $e$ is not a bridge and an edge in some connected component $E_a$, then define $s_e = s \cup \{ e \}$ and $c_e(E_a \cup \{ e \}) = c(E_a)$. \\
Let $C^i(G) = \oplus_{j \geq 0} C^{i,j}(G)$ and $\partial_{G}^{i} = \oplus_{j \geq 0} \partial_{G}^{i,j}$. Notice that the differential satisfies the property $\partial_{G}^{i+1} \partial_{G}^{i} = 0$, and thus $\mathcal{C}(G) = (C^i(G), \partial_{G}^i)$ is a chain complex. With the above notations the group
\begin{equation}
    H^i(G) = \frac{\ker \left( \partial_{G}^i : C^i(G) \to C^{i+1}(G) \right)}{\ima \left( \partial_{G}^{i-1} : C^{i-1}(G) \to C^i(G) \right)}
\end{equation}
is called the \textbf{graph homology} or \textbf{chromatic (graph) homology}. In the present paper we call it simply \textbf{chromatic homology}.
For an enhanced state $S = (s,c)$ of $G/e$, let $\tilde{s} = s \cup \{ e \}$ and $\tilde{c}$ be coloring of components of $[G : \tilde{s}]$. Then, by defining a map $\alpha^{i-1,j}(S) = (\tilde{s}, \tilde{c})$ and extending it linearly, we obtain a homomorphism $\alpha^{i-1,j} : C^{i-1, j}(G/e) \to C^{i,j}(G)$.

The following theorem gives a basic property of the chromatic homology.

\begin{thm}[Theorem 2.11 \cite{HY}]
    The Euler characteristic of the chromatic homology is equal to the chromatic polynomial of the graph evaluated at $\lambda = 1 + q$
\end{thm}

For an enhanced state $S = (s,c)$ of $G$ define a map $\beta^{i,j}: C^{i,j}(G) \to C^{i,j}(G-e)$ in such a way that if $e \notin s$, then $\beta^{i,j}(S) = S$, and if $e \in s$, then $\beta^{i,j}(S) = 0$. Again, by extending the map $\beta^{i,j}$ linearly we obtain the homomorphism $\beta^{i,j}: C^{i,j}(G) \to C^{i,j}(G-e)$. By summing over $j$ we have homomorphisms $\alpha^i: C^{i-1}(G/e) \to C^{i}(G)$ and $\beta^i:C^i(G) \to C^i(G-e)$, respectively. We abbreviate the maps by $\alpha$ and $\beta$.
The following lemma holds.
\begin{lem}[Lemma 3.1 \cite{HY}]
    $\alpha$ and $\beta$ are chain maps such that $0 \to C^{i-1,j}(G/e) \xrightarrow{\alpha} C^{i,j}(G) \xrightarrow{\beta} C^{i,j}(G-e) \to 0$ is a short exact sequence.
\end{lem}
By the Zig-Zag lemma the following theorem holds.
\begin{thm}[Theorem 3.2 \cite{HY}]\label{thm:exct_seq}
    Given a graph $G$ and an edge $e$ of $G$, for each $j$ there is a long exact sequence 
    \begin{align*}
         & 0 \to H^{0,j}(G) \xrightarrow{\beta^*} H^{0,j}(G-e) \xrightarrow{\gamma^*} H^{0,j}(G/e) \xrightarrow{\alpha^*} H^{1,j}(G) \xrightarrow{\beta^*} H^{1,j}(G-e) \xrightarrow{\gamma^*} \\
         & H^{1,j}(G/e) \to \dots \to \dots H^{i,j}(G) \xrightarrow{\beta^*} H^{i,j}(G-e) \xrightarrow{\gamma^*} H^{i,j}(G/e) \xrightarrow{\alpha^*} H^{i+1,j}(G) \to \dots
    \end{align*}
    If we sum over $j$, we have a degree-preserving long exact sequence:
    \begin{align*}
         & 0 \to H^{0}(G) \xrightarrow{\beta^*} H^{0}(G-e) \xrightarrow{\gamma^*} H^{0}(G/e) \xrightarrow{\alpha^*} H^{1}(G) \xrightarrow{\beta^*} H^{1}(G-e) \xrightarrow{\gamma^*} \\
         & H^{1}(G/e) \to \dots \to H^{i}(G) \xrightarrow{\beta^*} H^{i}(G-e) \xrightarrow{\gamma^*} H^{i}(G/e) \xrightarrow{\alpha^*} H^{i+1}(G) \to \dots
    \end{align*}
\end{thm}
In \cite{SS}, the splitting property of the chromatic homology is given for $i \geq 2$ as following.
\begin{lem}[Lemma 18 \cite{SS}]\label{lem:split_i2}
    Given a graph $G$ with $n$ vertices and an edge $e \in E(G)$ which is not a bridge, then for all $i \geq 2$,
    \begin{equation*}
        H^{i, n-i}(G) \simeq H^{i-1, n-i}(G/e) \oplus H^{i, n-i}(G-e).
    \end{equation*}
\end{lem}

\subsection{A combinatorial description of enhanced states}\label{subsec:combi_descr}
In this subsection, we introduce a combinatorial description of enhanced states. Let $S = (s,c)$ be an enhanced state of $G$ and $E_1, \dots, E_{d_1}, P_1, \dots, P_{d_2}$ be connected components of the spanning graph $[G : s]$, where each $E_h$, $h=1, \dots, d_1$ is a connected subgraph of $[G:s]$ with at least one edge, 
and each $P_k$, $k=1, \dots, d_2$ is a vertex. As an abuse of symbol let us denote the edge set $E(E_h)$ by $E_h$. Using this notation, we can describe enhanced states as follows. Order the components $E_h$, $h=1, \dots, d_1$ followed by $P_k$, $k=1, \dots, d_2$ and separate each component by the symbol ``$|$'' of the form $E_1 \bl \dots \bl E_{d_1} \bl P_1 \bl \dots \bl P_{d_2}$. Remark that we do not make particular assumptions about the ordering of the components. Put $x$ above the component $E_h$ or $P_k$ if its corresponding component is colored $x$. 

Let $G$ be a graph and $S = (s,c) \in C^{i,j}(G)$ be an enhanced state of $G$. For any components $K, K'$ of $[G : s]$, let us write as $K \sim_e K'$ if $e \in E(G) \setminus s$ connects them. For the components $E, E', P, P'$ of $[G:s]$ and an edge $e \in E(G) \setminus s$. We denote a new component obtained by adding the edge $e$ to the component(s) as follows.
\begin{align*}
    E^e                    & : \text{if $E \sim_e E$;}  \\
    \left( E  E' \right)^e & : \text{if $E \sim_e E'$;} \\
    \left( E  P \right)^e  & : \text{if $E \sim_e P$;}  \\
    \left( P  P' \right)^e & : \text{if $P \sim_e P'$}.
\end{align*}
Remark that for the new component the coloring $x$ is determined in a manner of the way explained in subsection \ref{subsec:ch_hlg} when $e$ is added. For fixed $1 \leq i_1 < \dots < i_t < \dots < i_p \leq d_1$, $1 \leq k_1 < \dots < k_{t'} < \dots < k_q \leq d_2$, $p+q=j$ let $S = (s,c) = E_1 \bl \dots \bl \overset{x}{E_{i_t}} \bl \dots \bl E_{d_1} \bl P_1 \bl \dots \bl \overset{x}{P_{k_{t'}}} \bl \dots \bl P_{d_2} \in C^{i,j}(G)$, where $\sum_{h=1}^{d_1} \# E_h = i$. For an edge $e \in E(G) \setminus \bigcup_{h=1}^{d_1} E_h$ we denote an enhanced state in which the edge $e$ is added to $S$ by $S \cup e$. More precisely, $S \cup e$ is one of the following:
\begin{equation*}
    \begin{array}{ll}
        \overset{(x)}{E_{a}^e} \coloneqq \left( E_1 \bl \dots \bl \overset{x}{E_{i_t}} \bl \dots \bl \overset{(x)}{E_a^e} \bl \dots \bl E_{d_1} \bl P_1 \bl \dots \bl \overset{x}{P_{k_{t'}}} \bl \dots \bl P_{d_2} \right)                                                 & \text{if $\overset{(x)}{E_a} \sim_e \overset{(x)}{E_a}$;}      \\
        \overset{(x)}{\left( E_a  E_b \right)^e} \coloneqq \left( E_1 \bl \dots \bl \overset{x}{E_{i_t}} \bl \dots \bl \overset{(x)}{\left( E_a  E_b \right)^e} \bl \dots \bl E_{d_1} \bl P_1 \bl \dots \bl \overset{x}{P_{k_{t'}}} \bl \dots \bl P_{d_2} \right)           & \text{if $\overset{(x)}{E_a} \sim_e \overset{(x)}{E_b}$;}      \\
        \overset{(x)}{\left( E_a  P_\alpha \right)^e} \coloneqq \left( E_1 \bl \dots \bl \overset{x}{E_{i_t}} \bl \dots \bl \overset{(x)}{\left( E_a  P_\alpha \right)^e} \bl \dots \bl E_{d_1} \bl P_1 \bl \dots \bl \overset{x}{P_{k_{t'}}} \bl \dots \bl P_{d_2} \right) & \text{if $\overset{(x)}{E_a} \sim_e \overset{(x)}{P_\alpha}$;} \\ 
        \overset{(x)}{\left( P_\alpha  P_\beta \right)^e} \coloneqq \left( E_1 \bl \dots \bl \overset{x}{E_{i_t}} \bl \dots \bl E_{d_1} \bl \overset{(x)}{\left( P_\alpha P_\beta \right)^e}  \bl P_1 \bl \dots \bl \overset{x}{P_{k_{t'}}} \bl \dots \bl P_{d_2} \right)   & \text{if $e = \left\{ P_\alpha, P_\beta \right\} \in E(G)$}.
    \end{array}
\end{equation*}
\begin{rem}
    If any two components $K, K'$ are connected by a bridge $e$, then $K$ and $K'$ are replaced by $\left( K K' \right)^e$.
    If $e$ connects two components that are both colored $x$, then we regard the enhanced state $S \cup e$ as 0. 
\end{rem}
Following figures express the corresponding enhanced states $S \cup e$. In the figures, each circle represents a connected component of the enhanced state $S$ and each point represents a vertex, both possibly with color $x$.
\begin{figure}[h]
    \begin{tikzpicture}
        \draw(0,0) circle (0.5cm);
        \draw(2,0) circle (0.5cm);
        \draw(4.0,0) circle (0.5cm);
        \draw(6,0) circle (0.5cm);
        
        \fill(7.0,0) circle(0.06 cm);
        \fill(8,0) circle(0.06 cm);
        \fill(9,0) circle(0.06 cm);
        
        \draw(0,-0.9) node{$E_1$};
        \draw(2,-0.9) node{$\overset{x}{E_{i_t}}$};
        \draw(4,-0.9) node{$\overset{(x)}{E_a^e}$};
        \draw(6,-0.9) node{$E_{d_1}$};
        
        \draw(7,-0.5) node{$P_1$};
        \draw(8,-0.5) node{$\overset{x}{P_{k_{t'}}}$};
        \draw(9,-0.5) node{$P_{d_2}$};
        
        \draw(1,0) node{$\cdots$};
        \draw(3,0) node{$\cdots$};
        \draw(5,0) node{$\cdots$};
        \draw(7.5,0) node{$\cdots$};
        \draw(8.5,0) node{$\cdots$};
        
        \draw (3.7,0.2) to [out=120,in=180] (4,0.8);
        \draw (4,0.8) to [out=0,in=60] (4.3,0.2);
        
        \draw(4,1) node{$e$};
        
    \end{tikzpicture}
    \caption{$\protect\overset{(x)}{E_{a}^e}$}
\end{figure}
\begin{figure}[h]
    \begin{tikzpicture}
        \draw(0,0) circle (0.5cm);
        \draw(2,0) circle (0.5cm);
        \draw(4,0) circle (0.5cm);
        \draw(6,0) circle (0.5cm);
        \draw(8,0) circle (0.5cm);
        
        \fill(9,0) circle(0.06 cm);
        \fill(10,0) circle(0.06 cm);
        \fill(11,0) circle(0.06 cm);
        
        \draw(0,-0.9) node{$E_1$};
        \draw(2,-0.9) node{$\overset{x}{E_{i_t}}$};
        \draw(4,-0.9) node{$\overset{(x)}{E_a}$};
        \draw(6,-0.9) node{$\overset{(x)}{E_b}$};
        \draw(8,-0.9) node{$E_{d_1}$};
        
        \draw(9,-0.5) node{$P_1$};
        \draw(10,-0.5) node{$\overset{x}{P_{k_{t'}}}$};
        \draw(11,-0.5) node{$P_{d_2}$};
        
        \draw(1,0) node{$\cdots$};
        \draw(3,0) node{$\cdots$};
        \draw(5,0) node{$\cdots$};
        \draw(7,0) node{$\cdots$};
        \draw(9.5,0) node{$\cdots$};
        \draw(10.5,0) node{$\cdots$};
        
        \draw (4,0.2) to [out=60,in=180] (5,0.8);
        \draw (5,0.8) to [out=0,in=120] (6,0.2);
        
        \draw(5,1) node{$e$};
        
    \end{tikzpicture}
    \caption{$\protect\overset{(x)}{(E_a E_b)^e}$}
\end{figure}
\begin{figure}[h]
    \begin{tikzpicture}
        \draw(0,0) circle (0.5cm);
        \draw(2,0) circle (0.5cm);
        \draw(4,0) circle (0.5cm);
        \draw(6,0) circle (0.5cm);
        
        \fill(7,0) circle(0.06 cm);
        \fill(8,0) circle(0.06 cm);
        \fill(9,0) circle(0.06 cm);
        \fill(10,0) circle(0.06 cm);
        
        \draw(0,-0.9) node{$E_1$};
        \draw(2,-0.9) node{$\overset{x}{E_{i_t}}$};
        \draw(4,-0.9) node{$\overset{(x)}{E_a}$};
        \draw(6,-0.9) node{$E_{d_1}$};
        
        \draw(7,-0.5) node{$P_1$};
        \draw(8,-0.5) node{$\overset{x}{P_{k_{t'}}}$};
        \draw(9,-0.5) node{$\overset{(x)}{P_{\alpha}}$};
        \draw(10,-0.5) node{$P_{d_2}$};
        
        \draw(1,0) node{$\cdots$};
        \draw(3,0) node{$\cdots$};
        \draw(5,0) node{$\cdots$};
        \draw(7.5,0) node{$\cdots$};
        \draw(8.5,0) node{$\cdots$};
        \draw(9.5,0) node{$\cdots$};
        
        \draw (4,0) to [out=60,in=180] (6.5,0.8);
        \draw (6.5,0.8) to [out=0,in=120] (9,0);
        
        \draw(6.5,1) node{$e$};
        
    \end{tikzpicture}
    \caption{$\protect\overset{(x)}{(E_a P_\alpha)^e}$}
\end{figure}
\begin{figure}[h]
    \begin{tikzpicture}
        \draw(0,0) circle (0.5cm);
        \draw(2,0) circle (0.5cm);
        \draw(4,0) circle (0.5cm);
        
        \fill(5,0) circle(0.06 cm);
        \fill(6,0) circle(0.06 cm);
        \fill(7,0) circle(0.06 cm);
        \fill(8,0) circle(0.06 cm);
        \fill(9,0) circle(0.06 cm);
        
        \draw(0,-0.9) node{$E_1$};
        \draw(2,-0.9) node{$\overset{x}{E_{i_t}}$};
        \draw(4,-0.9) node{$E_{d_1}$};
        
        \draw(5,-0.5) node{$P_1$};
        \draw(6,-0.5) node{$\overset{x}{P_{k_{t'}}}$};
        \draw(7,-0.5) node{$\overset{(x)}{P_{\alpha}}$};
        \draw(8,-0.5) node{$\overset{(x)}{P_{\beta}}$};
        \draw(9,-0.5) node{$P_{d_2}$};
        
        \draw(1,0) node{$\cdots$};
        \draw(3,0) node{$\cdots$};
        \draw(5.5,0) node{$\cdots$};
        \draw(6.5,0) node{$\cdots$};
        \draw(7.5,0) node{$\cdots$};
        \draw(8.5,0) node{$\cdots$};
        
        \draw (7,0) to [out=90,in=180] (7.5,0.8);
        \draw (7.5,0.8) to [out=0,in=90] (8,0);
        
        \draw(7.5,1) node{$e$};
        
    \end{tikzpicture}
    \caption{$\protect\overset{(x)}{(P_\alpha P_\beta)^e}$}
\end{figure}

For a component $E$ and two edges, $e,f$, of $G$ we denote a component obtained by adding the two edges $e, f$ to the same component $E$ by $E^{e,f}$. We denote the enhanced state obtained by adding distinct two edges $e, f$ in this order to $S$ by $S \cup e \cdot f$.
For $n \geq 3$, $((S \cup e_1) \cup e_2 ) \dots \cup e_n = S \cup e_1 \cdot e_2 \cdot \ldots \cdot e_n$ is determined inductively. 

\noindent For an enhanced state $S$ and distinct edges $e, f$ we give an anti-commutative structure as follows:
\begin{equation}\label{eq:anti_commutativity}
    S \cup e \cdot f = -S \cup f \cdot e.
\end{equation}
This is compatible with the fact that the changing the order in which the edges are added results in a change in the number of edges ordered before $e$ or $f$.

For a component $E$ we denote a full subgraph of $G$ with vertex set $V(E)$ by $F_{V(E)}$; that is, $F_{V(E)}$ is a subgraph $(V(F_{V(E)}), E(F_{V(E)}))$ of $G$ defined by $V(F_{V(E)}) = V(E)$, $E(F_{V(E)}) = \{ \{ a, b \} \bl a, b \in V(E), \{ a, b \} \in E(G) \}$. We denote the graph $F_{V(E)}$ and its edge set $E(F_{V(E)})$ by the same symbol $F_{E}$ for simplicity. For components $E_1$, $E_2$ define a new graph $E_1 \wedge E_2$ as a graph $(V(E_1 \wedge E_2), E(E_1 \wedge E_2))$, where $V(E_1 \wedge E_2) = V(E_1) \cup V(E_2)$, $E(E_1 \wedge E_2) = E_1 \cup E_2 \cup \{ \{ a, b \} \in E(G) \ | \ a \in V(E_1), b \in V(E_2) \}$. 

With the above notations we introduce a combinatorial description of a differential $\partial^{i,j}: C^{i,j}(G) \to C^{i+1,j}(G)$ as follows.

\begin{align}
     & \partial^{i,j} \left( E_1 \bl \dots \bl \overset{x}{E_{i_t}} \bl \dots \bl E_{d_1} \bl P_1 \bl \dots \bl \overset{x}{P_{k_{t'}}} \bl \dots \bl P_{d_2} \right)                                                                                                             \\
     & = \sum_{1 \leq a \leq d_1} \sum_{\substack{e \in F_{E_a} \setminus E_a}} (-1)^{n(e)} \overset{(x)}{E_{a}^e} \cup e + \sum_{1 \leq a<b \leq d_1} \sum_{\substack{e \in F_{E_a} \wedge F_{E_b} \setminus F_{E_a} \cup F_{E_b}}}(-1)^{n(e)} \overset{(x)}{(E_a E_b)^e} \cup e \\
     & + \sum_{\substack{1 \leq a \leq d_1                                                                                                                                                                                                                                        \\ 1 \leq \alpha \leq d_2}} \sum_{\substack{e \in F_{E_a} \wedge P_\alpha \setminus F_{E_a}}} (-1)^{n(e)} \overset{(x)}{(E_a P_\alpha)^e} \cup e +  \sum_{\substack{1 \leq \alpha < \beta \leq d_2 \\ e = \{ P_\alpha, P_\beta \}}} (-1)^{n(e)} \overset{(x)}{(P_\alpha P_\beta)^e} \cup e
\end{align}

where $n(e)$ is the number of edges ordered before $e$.

\begin{ex}
    Consider a complete graph $G = K_6$ with six vertices (see the left-hand image in Figure \ref{fig:ex_enh_{d_1}tate}) and its enhanced state $S = (s,c) \in C^{4,2}(G)$ (see the right-hand image in Figure \ref{fig:ex_enh_{d_1}tate}).
    \begin{figure}[h]
        \begin{minipage}[b]{0.45\linewidth}
            \centering
            \begin{tikzpicture}
                \draw 
                {
                (0,1) -- (-{sqrt(3)/2}, -1/2) 
                (0,1) -- (3/2,1) 
                (0,1) -- (3/2,-2) 
                (-{sqrt(3)/2}, -1/2) -- (0,-2) 
                (-{sqrt(3)/2}, -1/2) -- (3/2,1) 
                (-{sqrt(3)/2}, -1/2) -- (3/2,-2) 
                (0,1) -- (0,-2) 
                (3/2,1) -- (3/2,-2) 
                (0,-2) -- (3/2,-2)
                (0,-2) -- (3/2,1) 
                
                ({sqrt(3)/2 +3/2},-1/2) -- (0,1) 
                ({sqrt(3)/2 +3/2},-1/2) -- (-{sqrt(3)/2}, -1/2) 
                ({sqrt(3)/2 +3/2},-1/2) -- (0,-2)
                ({sqrt(3)/2 +3/2},-1/2) -- (3/2,-2) 
                ({sqrt(3)/2 +3/2},-1/2) -- (3/2,1) 
                
                (0,1.3) node{$1$}
                (-{sqrt(3)/2}-0.2, -0.5) node{$2$}
                (0,-2.3) node{$3$}
                (3/2,-2.3) node{$4$}
                ({sqrt(3)/2 +3/2+0.2},-0.5) node{$5$}
                (3/2,1.3) node{$6$}
                };
                
                \fill{
                (0,1) circle(0.06 cm)
                (-{sqrt(3)/2}, -1/2) circle(0.06 cm)
                (3/2,1) circle(0.06 cm)
                (0,-2) circle(0.06 cm)
                ({sqrt(3)/2 +3/2},-1/2) circle(0.06 cm)
                (3/2,-2) circle(0.06 cm)
                };
                
            \end{tikzpicture}
        \end{minipage}
        \begin{minipage}[b]{0.45\linewidth}
            \centering
            \begin{tikzpicture}
                \coordinate[label=right:$x$] (A) at (0,-1/2);
                \coordinate[label=right:1] (B) at (3/2,-1/2);
                \coordinate[label=right:$x$] (C) at (2.4,-1/2);
                \draw 
                {
                (0,1) -- (-{sqrt(3)/2}, -1/2)
                (-{sqrt(3)/2}, -1/2) -- (0,-2) 
                (0,1) -- (0,-2)
                (3/2,1) -- (3/2,-2)
                };
                
                \draw [loosely dashed] (0,1) -- (3/2,1);
                \draw [loosely dashed] (0,1) -- (3/2,-2);
                \draw [loosely dashed] (-{sqrt(3)/2}, -1/2) -- (3/2,1);
                \draw [loosely dashed] (-{sqrt(3)/2}, -1/2) -- (3/2,-2); 
                \draw [loosely dashed] (0,-2) -- (3/2,-2);
                \draw [loosely dashed] (0,-2) -- (3/2,1); 
                \draw [loosely dashed] ({sqrt(3)/2 +3/2},-1/2) -- (0,1); 
                \draw [loosely dashed] ({sqrt(3)/2 +3/2},-1/2) -- (-{sqrt(3)/2}, -1/2);
                \draw [loosely dashed] ({sqrt(3)/2 +3/2},-1/2) -- (0,-2);
                \draw [loosely dashed] ({sqrt(3)/2 +3/2},-1/2) -- (3/2,-2); 
                \draw [loosely dashed] ({sqrt(3)/2 +3/2},-1/2) -- (3/2,1);
                
                \draw (0,1.3) node{$1$};
                \draw (-{sqrt(3)/2}-0.2, -0.5) node{$2$};
                \draw (0,-2.3) node{$3$};
                \draw (3/2,-2.3) node{$4$};
                \draw ({sqrt(3)/2 +3/2},-0.8) node{$5$};
                \draw (3/2,1.3) node{$6$};
                
                \fill{
                (0,1) circle(0.06 cm)
                (-{sqrt(3)/2}, -1/2) circle(0.06 cm)
                (3/2,1) circle(0.06 cm)
                (0,-2) circle(0.06 cm)
                ({sqrt(3)/2 +3/2},-1/2) circle(0.06 cm)
                (3/2,-2) circle(0.06 cm)
                };
                
            \end{tikzpicture}
        \end{minipage}
        \caption{Complete graph $G = K_6$ and enhanced state $S = (s,c) \in C^{4,2}(G)$}\label{fig:ex_enh_{d_1}tate}
    \end{figure}\\
    The enhanced state $S$ can be written as 
    \begin{equation*}
        \overset{x}{\{ 1, 2 \}, \{ 1, 3 \}, \{ 2, 3 \}} \bl \{ 4, 6 \} \bl \overset{x}{\{ 5 \}} .
    \end{equation*}
    
    In this example, the differential $\partial^{4,2}$ would be calculated as follows.
    \begin{align*}
         & \partial^{4,2} \left( \overset{x}{\{ 1, 2 \}, \{ 1, 3 \}, \{ 2, 3 \}} \bl \{ 4, 6 \} \bl \overset{x}{\{ 5 \}} \right)                                                                                                                                                                                                                                                      \\
         & = \left( \overset{x}{\{ 1, 2 \}, \{ 1, 3 \}, \{ 2, 3 \}} \bl \{ 4, 6 \} \bl \overset{x}{\{ 5 \}} \right) \cup \{ 1, 4 \} + \left( \overset{x}{\{ 1, 2 \}, \{ 1, 3 \}, \{ 2, 3 \}} \bl \{ 4, 6 \} \bl \overset{x}{\{ 5 \}} \right) \cup \{ 1, 6 \} - \left( \overset{x}{\{ 1, 2 \}, \{ 1, 3 \}, \{ 2, 3 \}} \bl \{ 4, 6 \} \bl \overset{x}{\{ 5 \}} \right) \cup \{ 2, 4 \} \\
         & - \left( \overset{x}{\{ 1, 2 \}, \{ 1, 3 \}, \{ 2, 3 \}} \bl \{ 4, 6 \} \bl \overset{x}{\{ 5 \}} \right) \cup \{ 2, 5 \} - \left( \overset{x}{\{ 1, 2 \}, \{ 1, 3 \}, \{ 2, 3 \}} \bl \{ 4, 6 \} \bl \overset{x}{\{ 5 \}} \right) \cup \{ 3, 4 \} - \left( \overset{x}{\{ 1, 2 \}, \{ 1, 3 \}, \{ 2, 3 \}} \bl \{ 4, 6 \} \bl \overset{x}{\{ 5 \}} \right) \cup \{ 3, 6 \} \\
         & - \left( \overset{x}{\{ 1, 2 \}, \{ 1, 3 \}, \{ 2, 3 \}} \bl \{ 4, 6 \} \bl \overset{x}{\{ 5 \}} \right) \cup \{ 4, 5 \} + \left( \overset{x}{\{ 1, 2 \}, \{ 1, 3 \}, \{ 2, 3 \}} \bl \{ 4, 6 \} \bl \overset{x}{\{ 5 \}} \right) \cup \{ 5, 6 \}                                                                                                                          \\
         & = 
        \overset{x}{\{ 1,2 \}, \{ 1, 3 \}, \{ 1, 4 \}, \{ 2, 3 \}, \{ 4, 6 \}} \bl \overset{x}{\{ 5 \}} + \overset{x}{\{ 1, 2 \}, \{ 1, 3 \}, \{ 1, 6 \}, \{ 2, 3 \},\{ 4, 6 \}} \bl \overset{x}{\{ 5 \}} - \overset{x}{\{ 1, 2 \}, \{ 1, 3 \}, \{ 2, 3 \}, \{ 2, 4 \}, \{ 4, 6 \}} \bl \overset{x}{\{ 5 \}}                                                                          \\
         & - \overset{x}{\{ 1, 2 \},\{ 1, 3 \}, \{ 2, 3 \}, \{ 2 ,6 \}, \{ 4, 6 \}} \bl \overset{x}{\{ 5 \}} - \overset{x}{\{ 1, 2 \}, \{ 1, 3 \}, \{ 2, 3 \}, \{ 3, 4 \}, \{ 4, 6 \}} \bl \overset{x}{\{ 5 \}} - \overset{x}{\{ 1, 2 \}, \{ 1, 3 \}, \{ 2, 3 \}, \{ 3, 6 \}, \{ 4, 6 \}} \bl \overset{x}{\{ 5 \}}                                                                    \\
         & - \overset{x}{\{ 1, 2 \}, \{ 1, 3 \}, \{ 2, 3 \}} \bl \overset{x}{\{ 4, 5 \}, \{ 4, 6 \}} + \overset{x}{\{ 1, 2 \}, \{ 1, 3 \}, \{ 2, 3 \}} \bl \overset{x}{\{ 4, 6 \}, \{ 5, 6 \}}.
    \end{align*}
\end{ex}

Let $\mathcal{A}_m$ be the $\mathbb{Z}$-algebra defined by $\mathbb{Z}[x]/(x^m)$. The following lemma plays an important role for our calculation.
\begin{lem}[Corollary 13 \cite{H-GPR}]\label{lem:quote}
    Let $G$ be a $\mu$-component graph with $n$ vertices. If $G$ has no isolated vertices, then
    
    \begin{align*}
         & (1) : H_{\mathcal{A}_m}^{i,j}(G) \neq 0 \Rightarrow \begin{cases}
            0 \leq i \leq n - 2 \mu \\
            i + j \geq n - \mu      \\
            (m - 1)i + j \leq (m - 1)n.
        \end{cases}      \\
         & (2) : Tor(H_{\mathcal{A}_m}^{i,j}(G)) \neq 0 \Rightarrow \begin{cases}
            1 \leq i \leq n - 2 \mu \\
            i + j \geq n + 1 - \mu  \\
            (m - 1)i + j \leq (m - 1)n.
        \end{cases}
    \end{align*}
\end{lem}
\begin{rem}\label{rem:focuspart}
    In the present paper we assume the graph to be 1-component and the $\mathbb{Z}$-algebra to be $\mathcal{A}_2$. By Lemma \ref{lem:quote} it suffices to focus on $i,j$ such that $i+j = n-1, n$. In particular, the torsions possibly appear only $i,j$ such that $i+j = n$.
\end{rem}

In the rest of the paper, we assume the graph $G = K_{n-1}^m$ defined as following. For $n \geq 4$ consider the complete graph $K_{n-1}$ and a vertex $\{ n \}$ which is isolated from $K_{n-1}$. For fixed $m$, $1 \leq m \leq n-1$, define $S_m$ as a set $\{ \{ t, n \} \in E(K_{n-1} \wedge \{ n \}) \mid 1 \leq t \leq m \}$. Then, define $K_{n-1}^m = (V(K_{n-1}) \cup \{ n \}, E(K_{n-1}) \cup S_m)$. \\

\section{Construction of cocycles of the chromatic homology}\label{sec:main}
In this section we give generators of the cocycle and show a splitting property of the chromatic homology for the graph $G = K_{n-1}^m$ using the combinatorial description of enhanced states. The following proposition gives the generators of $\ker (\partial_G^{i,j} : C^{i,j}(G) \to C^{i+1,j}(G))$ when $i + j = n$.
\begin{prop}\label{prop:generator}
    Let $G = K_{n-1}^m$ be the graph defined in the previous section. The generators of $\ker (\partial_G^{i, n-i} : C^{i, n-i}(G) \to C^{i+1, n-i}(G))$ are given as follows.
    
    (I) $\ker \partial_G^{0, n} = \left< \overset{x}{P_1} \bl \dots \bl \overset{x}{P_n} \right>$.
    
    (II) $\ker \partial_G^{1,n-1} = \left< \overset{x}{E} \bl \overset{x}{P_1} \bl \dots \bl \overset{x}{P_{n-2}} ; E \subset E(G), \# E = 1 \right>$.
    
    (III) For $i \geq 2$,
    \begin{align*}
        \ker \partial_G^{i, n-i} & =  \left< \overset{x}{E_1} \bl \dots \bl \overset{x}{E_{i}} \bl \overset{x}{P_1} \bl \dots \bl \overset{x}{P_{n - 2i}} ; E_h \subset E(G), \# E_h = 1, h= 1, \dots, i, n - 2i \geq 0 \right>                                       \\
                                 & \cup \left< \sum_{e \in E(G) \setminus \bigcup_{h=1}^{d_1} E_h} \left( E_1 \bl \dots \bl \overset{x}{E_{i_{t}}} \bl \dots \bl E_{d_1} \bl P_1 \bl \dots \bl \overset{x}{P_{k_{t'}}} \bl \dots \bl P_{d_2} \right) \cup e ; \right. \\
                                 & E_h \subset E(G), h = 1, \dots, d_1, \begin{cases}
            1 \leq i_1< \dots < i_{t} < \dots < i_p \leq d_1,   \\
            1 \leq k_1< \dots < k_{t'}< \dots < k_{q} \leq d_2, \\
            p + q = n-i,
        \end{cases}
        \left. \begin{cases}
            d_2 = n - \sum_{h=1}^{d_1} \# V(E_h), \\
            \sum_{h=1}^{d_1} \# E_h = i-1
        \end{cases}
        \right>
    \end{align*}
\end{prop}
\begin{proof}
    To begin with let us check the generators in the list are actually in $\ker \partial^{i,n-i}$. Since the cases (I), (II) are easy to check, we only prove the case (III). It is obvious that
    \begin{align*}
        \partial_G^{i, n-i} \left( \overset{x}{E_1} \bl \dots \bl \overset{x}{E_{i}} \bl \overset{x}{P_1} \bl \dots \bl \overset{x}{P_{n - 2i}} \right) = 0,
    \end{align*}
    when $n - 2i \geq 0$ and $\# E_h = 1$, $h = 1, \dots, i$, since edges adding to the enhanced state $\overset{x}{E_1} \bl \dots \bl \overset{x}{E_{i}} \bl \overset{x}{P_1} \bl \dots \bl \overset{x}{P_{n - 2i}}$ are all bridges connecting two components labeled by $x$. 
    
    Next, let us compute 
    \begin{align*}
        \partial_G^{i,n-i} \left( \sum_{e \in E(G) \setminus \bigcup_{h=1}^{d_1} E_h} \left( E_1 \bl \dots \bl \overset{x}{E_{i_{t}}} \bl \dots \bl E_{d_1} \bl P_1 \bl \dots \bl \overset{x}{P_{k_{t'}}} \bl \dots \bl P_{d_2} \right) \cup e \right).
    \end{align*}
    Remark that we have
    \begin{align}
         & \sum_{e \in E(G) \setminus \bigcup_{h=1}^{d_1} E_h} \left( E_1 \bl \dots \bl \overset{x}{E_{i_{t}}} \bl \dots \bl E_{d_1} \bl P_1 \bl \dots \bl \overset{x}{P_{k_{t'}}} \bl \dots \bl P_{d_2} \right) \cup e                                                                                  \\
         & = \sum_{1 \leq a \leq d_1} \sum_{\substack{e \in F_{E_a} \setminus E_a}} \overset{(x)}{E_{a}^e} \cup e + \sum_{1 \leq a<b \leq d_1} \sum_{\substack{e \in F_{E_a} \wedge F_{E_b}                    \setminus F_{E_a} \cup F_{E_b}}} \overset{(x)}{(E_a  E_b)^e} \cup e \label{eq:boundary_1} \\
         & + \sum_{\substack{1 \leq a \leq d_1                                                                                                                                                                                                                                                           \\ 1 \leq \alpha \leq d_2}} \sum_{\substack{e \in F_{E_a} \wedge P_\alpha \setminus F_{E_a}}} \overset{(x)}{(E_a P_\alpha)^e} \cup e +  \sum_{\substack{1 \leq \alpha < \beta \leq d_2 \\ e = \{ P_\alpha, P_\beta \}}}  \overset{(x)}{(P_\alpha P_\beta)^e} \cup e \label{eq:boundary_2}.
    \end{align}
    Let us compute the boundary of each term of (\ref{eq:boundary_1}) and (\ref{eq:boundary_2}). In the rest of this computation, we omit the color $x$ for simplicity, which does not affect any of the computations. 
    \begin{align}
         & \partial^{i,n-i} \left( \sum_{1 \leq a \leq d_1} \sum_{\substack{e \in F_{E_a} \setminus E_a}} \overset{(x)}{E_{a}^e} \cup e \right)  = \sum_{1 \leq a \leq d_1} \sum_{\substack{e \in F_{E_a} \setminus E_a}} \partial^{i,n-i} \left( \overset{(x)}{E_{a}^e} \cup e \right) \\
         & = \sum_{\substack {1 \leq a,b \leq d_1                                                                                                                                                                                                                                       \\ a \neq b}} \sum_{\substack{e \in F_{E_a} \setminus E_a \\ f \in F_{E_b} \setminus E_b}} (-1)^{n(f)} \left( E_a^e \bl E_b^f \right) \cup e \cdot f \label{eq:I-I-I} \\
         & + \sum_{1 \leq a \leq d_1} \sum_{\substack{e, f \in F_{E_a} \setminus E_a                                                                                                                                                                                                    \\ e \neq f}} (-1)^{n(f)} \left( E_a^{e,f} \bl \right) \cup e \cdot f \label{eq:I-I-II}                                                                                                                                                                                                                                                            \\
         & + \sum_{\substack{1 \leq a,b,c \leq d_1                                                                                                                                                                                                                                      \\ b<c \\ b, c \neq a}} \sum_{\substack{e \in F_{E_a} \setminus E_a \\ f \in F_{E_b} \wedge F_{E_c} \setminus F_{E_b} \cup F_{E_c}}} (-1)^{n(f)} \left( E_a^e \bl \left( E_b  E_c \right)^f \right) \cup e \cdot f \label{eq:I-II-I} \\
         & + \sum_{\substack{1 \leq a,b \leq d_1                                                                                                                                                                                                                                        \\ a \neq b}} \sum_{\substack{e \in F_{E_a} \setminus E_a \\ f \in F_{E_a} \wedge F_{E_b} \setminus F_{E_a} \cup F_{E_b}}} (-1)^{n(f)} \left( \left( E_a^e  E_b \right)^f \bl \right) \cup e \cdot f \label{eq:I-II-II} \\
         & + \sum_{\substack{1 \leq a,b \leq d_1                                                                                                                                                                                                                                        \\ a \neq b \\ 1 \leq \alpha \leq d_2}} \sum_{\substack{e \in F_{E_a} \setminus E_a \\ f \in F_{E_b} \wedge P_\alpha \setminus F_{E_\alpha}}} (-1)^{n(f)} \left( E_a^e \bl (E_b  P_\alpha)^f \right) \cup e \cdot f \label{eq:I-III-I} \\
         & + \sum_{\substack{1 \leq a \leq d_1                                                                                                                                                                                                                                          \\ 1 \leq \alpha \leq d_2}} \sum_{\substack{e \in F_{E_a} \setminus E_a \\ f \in F_{E_a} \wedge P_\alpha \setminus F_{E_a}}} (-1)^{n(f)} \left( (E_a^e  P_\alpha)^f \bl \right) \cup e \cdot f                                                     \label{eq:I-III-II} \\
         & +  \sum_{\substack{1 \leq a \leq d_1                                                                                                                                                                                                                                         \\ 1 \leq \alpha < \beta \leq d_2 \\                                                                                                                                                                                                                               }} \sum_{\substack{e \in F_{E_a} \setminus E_a \\ f = \{ P_\alpha, P_\beta \}}} (-1)^{n(f)} \left( E_a^e \bl \left( P_\alpha P_\beta \right)^f \right) \cup e \cdot f \label{eq:I-IV}
    \end{align}
    
    \begin{align}
         & \partial^{i,n-i} \left( \sum_{1 \leq a<b \leq d_1} \sum_{\substack{e \in F_{E_a} \wedge F_{E_b}  \setminus F_{E_a} \cup F_{E_b}}} \left( E_a  E_b \right)^e \cup e \right) = \sum_{1 \leq a<b \leq d_1} \sum_{\substack{e \in F_{E_a} \wedge F_{E_b} \setminus F_{E_a} \cup F_{E_b}}} \partial^{i,n-i} \left( \left( E_a E_b \right)^e \cup e \right)                                                                                   \\
         & = \sum_{\substack{1 \leq a<b \leq d_1                                                                                                                                                                                                                                                                                                                                                                                                   \\ 1 \leq c \leq d_1 \\ c \neq a,b}} \sum_{\substack{e \in F_{E_a} \wedge F_{E_b}                      \setminus F_{E_a} \cup F_{E_b} \\ f \in F_{E_c} \setminus E_c}} (-1)^{n(f)} \left( \left( E_a E_b \right)^e \bl E_c^f \right) \cup e \cdot f  \label{eq:II-I-I} \\
         & + \sum_{1 \leq a<b \leq d_1} \sum_{\substack{e \in F_{E_a} \wedge F_{E_b}                                                                                                                                                                                                                                                                                                                                \setminus F_{E_a} \cup F_{E_b} \\ f \in F_{E_b} \setminus E_b}} (-1)^{n(f)} \left( \left( E_a E_b^f \right)^{e} \bl \right) \cup e \cdot f  \label{eq:II-I-II-I} \\
         & + \sum_{1 \leq a<b \leq d_1} \sum_{\substack{e \in F_{E_a} \wedge F_{E_b}                                                                                                                                                                                                                                                                                                \setminus F_{E_a} \cup F_{E_b}                                 \\ f \in F_{E_a} \setminus E_a}} (-1)^{n(f)} \left( \left( E_a^f  E_b \right)^e \bl \right) \cup e \cdot f  \label{eq:II-I-II-II} \\
         & + \sum_{1 \leq a<b \leq d_1} \sum_{\substack{e, f \in F_{E_a} \wedge F_{E_b}                                                                                                                                                                                                                                                                                             \setminus F_{E_a} \cup F_{E_b}                                 \\ e \neq f}} (-1)^{n(f)} \left( \left( E_a E_b \right)^{e,f} \bl \right) \cup e \cdot f  \label{eq:II-I-II-III} \\
         & + \sum_{\substack{1 \leq a<b \leq d_1                                                                                                                                                                                                                                                                                                                                                                                                   \\ 1 \leq c < d \leq d_1 \\ (a, b) \neq (c, d)}} \sum_{\substack{e \in F_{E_a} \wedge F_{E_b}                    \setminus F_{E_a} \cup F_{E_b} \\ f \in F_{E_c} \wedge F_{E_d}  \setminus F_{E_c} \cup F_{E_d}}} (-1)^{n(f)} \left( \left( E_a E_b \right)^e \bl \left( E_c E_d \right)^f \right) \cup e \cdot f  \label{eq:II-II-I} \\ 
         & + \sum_{\substack{1 \leq a<b \leq d_1                                                                                                                                                                                                                                                                                                                                                                                                   \\ 1 \leq c \leq d_1 \\ c \neq a, b}} \sum_{\substack{e \in F_{E_a} \wedge F_{E_b}           \setminus F_{E_a} \cup F_{E_b} \\ f \in F_{E_a} \wedge F_{E_c} \setminus F_{E_a} \cup F_{E_c}}} (-1)^{n(f)} \left( \left( E_a E_b \right)^e  \left( E_a E_c  \right)^f \bl \right) \cup e \cdot f  \label{eq:II-II-II-a} \\
         & + \sum_{\substack{1 \leq a<b \leq d_1                                                                                                                                                                                                                                                                                                                                                                                                   \\ 1 \leq c \leq d_1 \\ c \neq a, b}} \sum_{\substack{e \in F_{E_a} \wedge F_{E_b}           \setminus F_{E_a} \cup F_{E_b} \\ f \in F_{E_b} \wedge F_{E_c} \setminus F_{E_b} \cup F_{E_c}}} (-1)^{n(f)} \left( \left( E_a E_b \right)^e  \left( E_b E_c  \right)^f \bl \right) \cup e \cdot f  \label{eq:II-II-II-b} \\
         & + \sum_{\substack{1 \leq a<b \leq d_1                                                                                                                                                                                                                                                                                                                                                                                                   \\ 1 \leq c \leq d_1 \\ c \neq a,b \\ 1 \leq \alpha \leq d_2}} \sum_{\substack{e \in F_{E_a} \wedge F_{E_b}                      \setminus F_{E_a} \cup F_{E_b} \\ f \in F_{E_c} \wedge P_\alpha \setminus F_{E_c}}} (-1)^{n(f)} \left( \left( E_a E_b \right)^e \bl \left( E_c  P_\alpha \right)^f \right) \cup e \cdot f  \label{eq:II-III-I} \\
         & + \sum_{\substack{1 \leq a<b \leq d_1                                                                                                                                                                                                                                                                                                                                                                                                   \\ 1 \leq \alpha \leq d_2}} \sum_{\substack{e \in F_{E_a} \wedge F_{E_b}                      \setminus F_{E_a} \cup F_{E_b} \\ f \in F_{E_a} \wedge P_\alpha \setminus F_{E_a}}} (-1)^{n(f)} \left( \left( E_a E_b \right)^e \left( E_a P_\alpha \right)^f \bl \right) \cup e \cdot f  \label{eq:II-III-II-a} \\
         & + \sum_{\substack{1 \leq a<b \leq d_1                                                                                                                                                                                                                                                                                                                                                                                                   \\ 1 \leq \alpha \leq d_2}} \sum_{\substack{e \in F_{E_a} \wedge F_{E_b}                      \setminus F_{E_a} \cup F_{E_b} \\ f \in F_{E_b} \wedge P_\alpha \setminus F_{E_b}}} (-1)^{n(f)} \left( \left( E_a E_b \right)^e \left( E_b P_\alpha \right)^f \bl \right) \cup e \cdot f  \label{eq:II-III-II-b} \\
         & + \sum_{\substack{1 \leq a<b \leq d_1                                                                                                                                                                                                                                                                                                                                                                                                   \\ 1 \leq \alpha<\beta \leq d_2}} \sum_{\substack{e \in F_{E_a} \wedge F_{E_b}                      \setminus F_{E_a} \cup F_{E_b} \\ f = \{ P_\alpha, P_\beta \}}} (-1)^{n(f)} \left( \left( E_a E_b \right)^e \bl \left( P_\alpha P_\beta \right)^f \right) \cup e \cdot f \label{eq:II-IV}
    \end{align}
    
    \begin{align}
         & \partial^{i,n-i} \left( \sum_{\substack{1 \leq a \leq d_1 \\ 1 \leq \alpha \leq d_2}} \sum_{\substack{e \in F_{E_a} \wedge P_\alpha \setminus F_{E_a}}} \left( E_a P_\alpha \right)^e  \cup e \right)  \\
         & = \sum_{\substack{1 \leq a \leq d_1                       \\ 1 \leq \alpha \leq d_2}} \sum_{\substack{e \in F_{E_a} \wedge P_\alpha \setminus F_{E_a}}} \partial^{i,n-i} \left( \left( E_a P_\alpha \right)^e \cup e \right) \\
         & = \sum_{\substack{1 \leq a,b \leq d_1                     \\ a \neq b                  \\ 1 \leq \alpha \leq d_2 }} \sum_{\substack{e \in F_{E_a} \wedge P_\alpha \setminus F_{E_a} \\ f \in F_{E_b} \setminus E_b}} (-1)^{n(f)} \left( \left( E_a P_\alpha \right)^e \bl E_b^f \right) \cup e \cdot f \label{eq:III-I-I}  \\
         & + \sum_{\substack{1 \leq a \leq d_1                       \\ 1 \leq \alpha \leq d_2}} \sum_{\substack{e \in E_a \wedge P_\alpha \setminus F_{E_a} \\ f \in F_{E_a} \setminus E_a}} (-1)^{n(f)} \left( \left( E_a^f P_\alpha \right)^e \bl \right) \cup e \cdot f \label{eq:wasure}\\
         & + \sum_{\substack{1 \leq a \leq d_1                       \\ 1 \leq \alpha \leq d_2}} \sum_{\substack{e, f \in F_{E_a} \wedge P_\alpha \setminus F_{E_a} \\ e \neq f}} (-1)^{n(f)} \left( \left( E_a P_\alpha \right)^{e,f} \bl \right) \cup e \cdot f \label{eq:III-I-II} \\
         & + \sum_{\substack{1 \leq a \leq d_1                       \\ 1 \leq b < c \leq d_1 \\ b,c \neq a                  \\ 1 \leq \alpha \leq d_2}} \sum_{\substack{e \in F_{E_a} \wedge P_\alpha \setminus F_{E_a} \\ f \in F_{E_b} \wedge F_{E_c} \setminus F_{E_b} \cup F_{E_c} }} (-1)^{n(f)} \left( \left( E_a P_\alpha \right)^e \bl \left( E_b E_c \right)^f \right) \cup e \cdot f \label{eq:III-II-I} \\
         & + \sum_{\substack{1 \leq a,b \leq d_1                     \\ a \neq b               \\ 1 \leq \alpha \leq d_2}} \sum_{\substack{e \in F_{E_a} \wedge P_\alpha \setminus F_{E_a} \\ f \in F_{E_a} \wedge F_{E_b} \setminus F_{E_a} \cup F_{E_b}}} (-1)^{n(f)} \left( \left( E_a P_\alpha \right)^e \left( E_a E_b \right)^f \bl \right) \cup e \cdot f \label{eq:III-II-II} \\
         & + \sum_{\substack{1 \leq a, b \leq d_1                    \\ a \neq b \\ 1 \leq \alpha, \beta \leq d_2 \\ \alpha \neq \beta}} \sum_{\substack{e \in F_{E_a} \wedge P_\alpha \setminus F_{E_a} \\ f \in F_{E_b} \wedge P_\beta \setminus F_{E_b}}} (-1)^{n(f)} \left( \left( E_a P_\alpha \right)^e \bl \left( E_b P_\beta \right)^f \right) \cup e \cdot f \label{eq:III-III-I} \\
         & + \sum_{\substack{1 \leq a \leq d_1                       \\ 1 \leq \alpha, \beta \leq d_2 \\ \alpha \neq \beta}} \sum_{\substack{e \in F_{E_a} \wedge P_\alpha \setminus F_{E_a} \\ f \in E_a \wedge P_\beta \setminus F_{E_a}}} (-1)^{n(f)} \left( \left( E_a P_\alpha \right)^e \left( E_a P_\beta \right)^f \bl \right) \cup e \cdot f \label{eq:III-III-II} \\
         & + \sum_{\substack{1 \leq a \leq d_1                       \\ 1 \leq \alpha \leq d_2 \\ 1 \leq \beta < \gamma \leq d_2 \\ \beta, \gamma \neq \alpha}} \sum_{\substack{e \in F_{E_a} \wedge P_\alpha \setminus F_{E_a} \\ f = \{ P_\beta, P_\gamma \}}} (-1)^{n(f)} \left( \left( E_a P_\alpha \right)^e \bl \left( P_\beta P_\gamma \right)^f  \right) \cup e \cdot f \label{eq:III-IV-I} \\
         & + \sum_{\substack{1 \leq a \leq d_1                       \\ 1 \leq \alpha<\beta \leq d_2}} \sum_{\substack{e \in F_{E_a} \wedge P_\alpha \setminus F_{E_a} \\ f = \{ P_\alpha, P_\beta \}}} (-1)^{n(f)} \left( \left( E_a P_\alpha \right)^e \left( P_\alpha P_\beta \right)^f  \bl \right) \cup e \cdot f \label{eq:III-IV-II} 
    \end{align}
    
    \begin{align}
         & \partial^{i,n-i} \left( \sum_{\substack{1 \leq \alpha < \beta \leq d_2 \\ e = \{ P_\alpha, P_\beta \}}} \left( P_\alpha P_\beta \right)^e \cup e \right)                                                                                                                                        \\
         & = \sum_{\substack{1 \leq \alpha < \beta \leq d_2                       \\ e = \{ P_\alpha, P_\beta \} \\}} \partial^{i,n-i} \left( \left( P_\alpha P_\beta \right)^e \cup e \right)                                                                                                                                      \\
         & = \sum_{\substack{1 \leq a \leq d_1                                    \\ 1 \leq \alpha < \beta \leq d_2}} \sum_{\substack{ e = \{ P_\alpha, P_\beta \} \\ f \in F_{E_a} \setminus E_a}} (-1)^{n(f)} \left( E_a^f \bl \left( P_\alpha P_\beta \right)^e \right) \cup e \cdot f \label{eq:IV-I}                                         \\ 
         & + \sum_{\substack{1 \leq a<b \leq d_1                                  \\ 1 \leq \alpha < \beta \leq d_2}} \sum_{\substack{e = \{ P_\alpha, P_\beta \} \\ f \in F_{E_a} \wedge F_{E_b}                                                                                                                                          \setminus F_{E_a} \cup F_{E_b}}} (-1)^{n(f)} \left( \left( E_a E_b \right)^f \bl \left( P_\alpha P_\beta \right)^e \right) \cup e \cdot f  \label{eq:IV-II} \\
         & + \sum_{\substack{1 \leq a \leq d_1                                    \\ 1 \leq \alpha < \beta \leq d_2 \\ 1 \leq \gamma \leq d_1 \\ \alpha, \beta \neq \gamma}}  \sum_{\substack{e = \{ P_\alpha, P_\beta \} \\f \in F_{E_a} \wedge P_\alpha \setminus F_{E_a}}} (-1)^{n(f)} \left( \left( E_a P_\gamma \right)^f \bl \left( P_\alpha P_\beta \right)^e \right) \cup e \cdot f  \label{eq:IV-III-I} \\
         & + \sum_{\substack{1 \leq a \leq d_1                                    \\ 1 \leq \alpha < \beta \leq d_2}} \sum_{\substack{e = \{ P_\alpha, P_\beta \} \\ f \in F_{E_a} \wedge P_\alpha \setminus F_{E_a}}} (-1)^{n(f)} \left( \left( E_a  P_\alpha \right)^f \left( P_\alpha P_\beta \right)^e \bl \right) \cup e \cdot f \label{eq:IV-III-II}    \\
         & + \sum_{\substack{1 \leq \alpha < \beta \leq d_2                       \\ 1 \leq \gamma < \delta \leq d_2 \\ (\alpha, \beta) \neq (\gamma, \delta) \\e = \{ P_\alpha, P_\beta \}, f = \{ P_\gamma, P_\delta \}}} (-1)^{n(f)} \left( \left( P_\alpha P_\beta \right)^e \bl \left( P_\gamma P_\delta \right)^f \right) \cup e \cdot f \label{eq:IV-IV-I} \\
         & + \sum_{\substack{1 \leq \alpha < \beta \leq d_2                       \\ 1 \leq \gamma \leq d_2 \\ \gamma \neq \alpha,\beta \\ e = \{ P_\alpha, P_\beta \}, f = \{ P_\alpha, P_\gamma \}}} (-1)^{n(f)} \left( \left( P_\alpha P_\beta \right)^e \left( P_\alpha P_\gamma \right)^f \bl \right) \cup e \cdot f \label{eq:IV-IV-II} \\
         & + \sum_{\substack{1 \leq \alpha < \beta \leq d_2                       \\ 1 \leq \gamma \leq d_2 \\ \gamma \neq \alpha,\beta \\ e = \{ P_\alpha, P_\beta \}, f = \{ P_\beta, P_\gamma \}}} (-1)^{n(f)} \left( \left( P_\alpha P_\beta \right)^e \left( P_\beta P_\gamma \right)^f \bl \right) \cup e \cdot f  \label{eq:IV-IV-III}
    \end{align}
    By the following computations we can see that 
    \begin{equation*}
        \partial^{i,n-i} \left( \sum_{e \in E(G) \setminus \bigcup_{h=1}^{d_1} E_h} \left( E_1 \bl \dots \bl \overset{x}{E_{i_{t}}} \bl \dots \bl E_{d_1} \bl P_1 \bl \dots \bl \overset{x}{P_{k_{t'}}} \bl \dots \bl P_{d_2} \right) \cup e  \right) = 0.
    \end{equation*}
    
    \begin{align*}
        (\ref{eq:I-I-I}) & = \sum_{\substack{1 \leq a<b \leq d_1}} \sum_{\substack{e \in F_{E_a} \setminus E_a \\ f \in F_{E_b} \setminus E_b}} (-1)^{n(f)} \left( \left( E_a^e \bl E_b^f \right) \cup e \cdot f + \left( E_a^e \bl E_b^f \right) \cup f \cdot e \right) \\
                         & = 0.
    \end{align*}
    
    \begin{align*}
        (\ref{eq:I-I-II}) & = \sum_{1 \leq a \leq d_1} \sum_{\substack{e \in F_{E_a} \setminus E_a \\ f \in F_{E_a} \setminus E_a \cup \{ e \}}} (-1)^{n(f)} \left( \left( E_a^{e,f} \bl \right) \cup e \cdot f + \left( E_a^{e,f} \bl \right) \cup f \cdot e \right) \\
                          & = 0.
    \end{align*}
    
    \begin{align*}
        (\ref{eq:I-II-I}) + (\ref{eq:II-I-I}) & = \sum_{1 \leq a < b < c \leq d_1} \sum_{\substack{e \in F_{E_a} \setminus E_a \\ f \in F_{E_b} \wedge F_{E_c} \setminus F_{E_b} \cup F_{E_c}}} (-1)^{n(f)} \left( \left( E_a^e \bl \left( E_b E_c \right)^f \right) \cup e \cdot f + \left( E_a^e \bl \left( E_b E_c \right)^f \right) \cup f \cdot e \right) \\
                                              & + \sum_{1 \leq b < a < c \leq d_1} \sum_{\substack{e \in F_{E_a} \setminus E_a \\ f \in F_{E_b} \wedge F_{E_c} \setminus F_{E_b} \cup F_{E_c}}} (-1)^{n(f)} \left( \left( E_a^e \bl \left( E_b E_c \right)^f \right) \cup e \cdot f + \left( E_a^e \bl \left( E_b E_c \right)^f \right) \cup f \cdot e \right) \\
                                              & + \sum_{1 \leq b < c < a \leq d_1} \sum_{\substack{e \in F_{E_a} \setminus E_a \\ f \in F_{E_b} \wedge F_{E_c} \setminus F_{E_b} \cup F_{E_c}}} (-1)^{n(f)} \left( \left( E_a^e \bl \left( E_b E_c \right)^f \right) \cup e \cdot f + \left( E_a^e \bl \left( E_b E_c \right)^f \right) \cup f \cdot e \right) \\
                                              & = 0
    \end{align*}
    
    \begin{align*}
        (\ref{eq:I-II-II}) + (\ref{eq:II-I-II-I}) + (\ref{eq:II-I-II-II}) & = \sum_{1 \leq a < b \leq d_1} \sum_{\substack{e \in F_{E_a} \setminus E_a \\ f \in F_{E_a} \wedge F_{E_b} \setminus F_{E_a} \cup F_{E_b}}} (-1)^{n(f)} \left( \left( \left( E_a^e E_b \right)^f \bl \right) \cup e \cdot f + \left( \left( E_a^e E_b \right)^f \bl \right) \cup f \cdot e \right) \\
                                                                          & + \sum_{1 \leq a < b \leq d_1} \sum_{\substack{e \in F_{E_a} \setminus E_a \\ f \in F_{E_a} \wedge F_{E_b} \setminus F_{E_a} \cup F_{E_b}}} (-1)^{n(f)} \left( \left( \left( E_a E_b^e \right)^f \bl \right) \cup e \cdot f + \left( \left( E_a E_b^e \right)^f \bl \right) \cup f \cdot e \right) \\
                                                                          & = 0
    \end{align*}
    \begin{align*}
        (\ref{eq:I-III-I}) + (\ref{eq:III-I-I}) & = \sum_{\substack{1 \leq a,b \leq d_1 \\ a \neq b \\ 1 \leq \alpha \leq d_2}} \sum_{\substack{e \in F_{E_a} \setminus E_a \\ f \in E_b \wedge P_\alpha \setminus F_{E_b}}} (-1)^{n(f)} \left( \left( E_a^e \bl (E_b  P_\alpha)^f \right) \cup e \cdot f + \left( E_a^e \bl (E_b  P_\alpha)^f \right) \cup f \cdot e \right) \\
                                                & = 0
    \end{align*}
    \begin{align*}
        (\ref{eq:I-III-II}) + (\ref{eq:wasure}) & = \sum_{\substack{1 \leq a \leq d_1         \\ 1 \leq \alpha \leq d_2}} \sum_{\substack{e \in F_{E_a} \setminus E_a \\ f \in F_{E_a} \wedge P_\alpha \setminus F_{E_a}}} (-1)^{n(f)} \left( \left( \left( E_a^e  P_\alpha \right)^f \bl \right) \cup e \cdot f + \left( \left( E_a^e  P_\alpha \right)^f \bl \right) \cup f \cdot e \right) \\
                                                & = 0                                        
    \end{align*}
    
    \begin{align*}
        (\ref{eq:I-IV}) + (\ref{eq:IV-I}) & = \sum_{\substack{1 \leq a \leq d_1 \\ 1 \leq \alpha < \beta \leq d_2}} \sum_{\substack{e \in F_{E_a} \setminus E_a \\ f = \{ P_\alpha, P_\beta \}}} (-1)^{n(f)} \left( \left( E_a^e \bl \left( P_\alpha P_\beta \right)^f \right) \cup e \cdot f + \left( E_a^e \bl \left( P_\alpha P_\beta \right)^f \right) \cup f \cdot e \right) \\
                                          & = 0
    \end{align*}
    \begin{align*}
        (\ref{eq:II-I-II-III}) & =  \sum_{1 \leq a < b \leq d_1} \sum_{\substack{e \in F_{E_a} \wedge F_{E_b} \setminus F_{E_a} \cup F_{E_b} \\ f \in F_{E_a} \wedge F_{E_b} \setminus F_{E_a} \cup F_{E_b} \cup \{ e \}}} \left( (-1)^{n(f)} \left( \left( E_a E_b \right)^{e,f} \bl \right) \cup e \cdot f +  \left( \left( E_a E_b \right)^{e,f} \bl \right) \cup f \cdot e \right) \\
                               & = 0
    \end{align*}
    \begin{align*}
        (\ref{eq:II-II-I}) & = \sum_{\substack{1 \leq a<b<c<d \leq d_1 }} \sum_{\substack{e \in F_{E_a} \wedge F_{E_b} \setminus E_a \cup E_b  \\ f \in F_{E_c} \wedge F_{E_d} \setminus E_c \cup E_d}} (-1)^{n(f)} \left( \left( \left( E_a E_b \right)^e \bl \left( E_c E_d \right)^f \right) \cup e \cdot f + \left( \left( E_a E_b \right)^e \bl \left( E_c E_d \right)^f \right) \cup f \cdot e \right) \\
                           & + \sum_{\substack{1 \leq a< c< b<d \leq d_1}} \sum_{\substack{e \in F_{E_a} \wedge F_{E_b} \setminus E_a \cup E_b \\  f \in F_{E_c} \wedge F_{E_d} \setminus E_c \cup E_d}} (-1)^{n(f)} \left( \left( \left( E_a E_b \right)^e \bl \left( E_c E_d \right)^f \right) \cup e \cdot f + \left( \left( E_a E_b \right)^e \bl \left( E_c E_d \right)^f \right) \cup f \cdot e \right) \\
                           & + \sum_{\substack{1 \leq c< a< b<d \leq d_1}} \sum_{\substack{e \in F_{E_a} \wedge F_{E_b} \setminus E_a \cup E_b \\  f \in F_{E_c} \wedge F_{E_d} \setminus E_c \cup E_d}} (-1)^{n(f)} \left( \left( \left( E_a E_b \right)^e \bl \left( E_c E_d \right)^f \right) \cup e \cdot f + \left( \left( E_a E_b \right)^e \bl \left( E_c E_d \right)^f \right) \cup f \cdot e \right) \\
                           & = 0
    \end{align*}
    \begin{align*}
        (\ref{eq:II-II-II-a}) + (\ref{eq:II-II-II-b}) & = \sum_{1 \leq a < b < c \leq d_1} \left( \sum_{\substack{e \in F_{E_a} \wedge F_{E_b} \setminus F_{E_a} \cup F_{E_b} \\ f \in F_{E_a} \wedge F_{E_c} \setminus F_{E_a} \cup F_{E_c}}} (-1)^{n(f)} \left( \left( \left( E_a E_b \right)^e \left( E_a E_c \right)^f \bl \right) \cup e \cdot f + \left( \left( E_a E_b \right)^e \left( E_a E_c \right)^f \bl \right) \cup f \cdot e \right) \right. \\
                                                      & + \sum_{\substack{e \in F_{E_b} \wedge F_{E_c} \setminus F_{E_b} \cup F_{E_c}                                         \\ f \in F_{E_a} \wedge F_{E_b} \setminus F_{E_a} \cup F_{E_b}}} (-1)^{n(f)} \left( \left( \left( E_a E_b \right)^e \left( E_a E_c \right)^f \bl \right) \cup e \cdot f + \left( \left( E_a E_b \right)^e \left( E_a E_c \right)^f \bl \right) \cup f \cdot e \right) \\
                                                      & + \left. \sum_{\substack{e \in F_{E_a} \wedge F_{E_c} \setminus F_{E_a} \cup F_{E_c}                                  \\ f \in F_{E_a} \wedge F_{E_b} \setminus F_{E_a} \cup F_{E_b}}} (-1)^{n(f)} \left( \left( \left( E_a E_b \right)^e \left( E_a E_c \right)^f \bl \right) \cup e \cdot f + \left( \left( E_a E_b \right)^e \left( E_a E_c \right)^f \bl \right) \cup f \cdot e \right) \right) \\
                                                      & = 0
    \end{align*}
    \begin{align*}
        (\ref{eq:II-III-I}) + (\ref{eq:III-II-I}) & = \sum_{\substack{1 \leq a < b < d_1 \\ 1 \leq c \leq d_1 \\ c \neq a, b \\ 1 \leq \alpha \leq d_2}} \sum_{\substack{e \in F_{E_a} \wedge F_{E_b} \setminus F_{E_a} \cup F_{E_b} \\ f \in F_{E_c} \wedge P_\alpha \setminus F_{E_c}}} (-1)^{n(f)} \left( \left( \left( E_a E_b \right)^e \bl \left( E_c P_\alpha \right)^f \right) \cup e \cdot f + \left( \left( E_a E_b \right)^e \bl \left( E_c P_\alpha \right)^f \right) \cup f \cdot e \right) \\
                                                  & = 0
    \end{align*}
    
    \begin{align*}
        (\ref{eq:II-III-II-a}) + (\ref{eq:II-III-II-b}) + (\ref{eq:III-II-II}) & = \sum_{\substack{1 \leq a<b \leq d_1                                                                    \\ 1 \leq \alpha \leq d_2}} \left( \sum_{\substack{e \in F_{E_a} \wedge F_{E_b}                     \setminus F_{E_a} \cup F_{E_b} \\ f \in F_{E_a} \wedge P_\alpha \setminus F_{E_a}}}  (-1)^{n(f)} \left( \left( \left( E_a  E_b \right)^e \left( E_a P_\alpha \right)^f \bl \right) \cup e \cdot f + \left( \left( E_a  E_b \right)^e \left( E_a P_\alpha \right)^f \bl \right) \cup f \cdot e \right) \right. \\
                                                                               & + \left. \sum_{\substack{e \in F_{E_a} \wedge F_{E_b}                     \setminus F_{E_a} \cup F_{E_b} \\ f \in F_{E_b} \wedge P_\alpha \setminus F_{E_b}}}  \left( \left( \left( E_a  E_b \right)^e \left( E_b P_\alpha \right)^f \bl \right) \cup e \cdot f + \left( \left( E_a  E_b \right)^e \left( E_b P_\alpha \right)^f \bl \right) \cup f \cdot e \right) \right) \\
                                                                               & = 0
    \end{align*}
    \begin{align*}
        (\ref{eq:II-IV}) + (\ref{eq:IV-II}) & = \sum_{\substack{1 \leq a<b \leq d_1 \\ 1 \leq \alpha < \beta \leq d_2}} \sum_{\substack{e \in F_{E_a} \wedge F_{E_b} \setminus F_{E_a} \cup F_{E_b} \\ f = \{ P_\alpha, P_\beta \}}} \left( \left( \left( E_a E_b \right)^e \bl  \left( P_\alpha P_\beta \right)^f \right) \cup e \cdot f + \left( \left( E_a E_b \right)^e \bl  \left( P_\alpha P_\beta \right)^f \right) \cup f \cdot e \right) \\
                                            & = 0
    \end{align*}
    \begin{align*}
        (\ref{eq:III-I-II}) & = \sum_{\substack{1 \leq a \leq d_1 \\ 1 \leq \alpha \leq d_2}} \sum_{\substack{e \in F_{E_a} \wedge P_\alpha \setminus F_{E_a} \\ f \in F_{E_a} \wedge P_\alpha \setminus F_{E_a} \cup \{ e \}}} (-1)^{n(f)} \left( \left( \left( E_a P_\alpha \right)^{e,f} \bl \right) \cup e \cdot f + \left( \left( E_a P_\alpha \right)^{e,f} \bl \right) \cup f \cdot e \right) \\
                            & = 0
    \end{align*}
    \begin{align*}
        (\ref{eq:III-III-I}) & = \sum_{\substack{1 \leq a < b \leq d_1 \\ 1 \leq \alpha < \beta \leq d_2}} \sum_{\substack{e \in E_a \wedge P_\alpha \setminus F_{E_a} \\ f \in E_b \wedge P_\beta \setminus F_{E_b}}} (-1)^{n(f)} \left( \left( \left( E_a P_\alpha \right)^e \bl \left( E_b P_\beta \right)^f \right) \cup e \cdot f + \left( \left( E_a P_\alpha \right)^e \bl \left( E_b P_\beta \right)^f \right) \cup f \cdot e \right) \\
                             & + \sum_{\substack{1 \leq a < b \leq d_1 \\ 1 \leq \beta < \alpha \leq d_2}} \sum_{\substack{e \in E_a \wedge P_\alpha \setminus F_{E_a} \\ f \in E_b \wedge P_\beta \setminus F_{E_b}}} (-1)^{n(f)} \left( \left( \left( E_a P_\alpha \right)^e \bl \left( E_b P_\beta \right)^f \right) \cup e \cdot f + \left( \left( E_a P_\alpha \right)^e \bl \left( E_b P_\beta \right)^f \right) \cup f \cdot e \right) \\
                             & = 0
    \end{align*}
    \begin{align*}
        (\ref{eq:III-III-II}) & = \sum_{\substack{1 \leq a \leq d_1 \\ 1 \leq \alpha < \beta \leq d_2}} \sum_{\substack{e \in F_{E_a} \wedge P_\alpha \setminus F_{E_a} \\ f \in F_{E_a} \wedge P_\beta \setminus F_{E_a}}} (-1)^{n(f)} \left( \left( \left( E_a P_\alpha \right)^e \left( E_a P_\beta \right)^f \bl \right) \cup e \cdot f + \left( \left( E_a P_\alpha \right)^e \left( E_a P_\beta \right)^f \bl \right) \cup f \cdot e \right) \\
                              & = 0
    \end{align*}
    
    \begin{align*}
        (\ref{eq:III-IV-I}) + (\ref{eq:IV-III-I}) & = \sum_{\substack{1 \leq a \leq d_1 \\ 1 \leq \alpha \leq d_2 \\ 1 \leq \beta<\gamma \leq d_2 \\ \beta, \gamma \neq \alpha}} \sum_{\substack{e \in F_{E_a} \wedge P_\alpha \setminus F_{E_a} \\ f = \{ P_\beta, P_\gamma \}}} (-1)^{n(f)} \left( \left( \left( E_a P_\alpha \right)^e \bl \left( P_\beta P_\gamma \right)^f \right) \cup e \cdot f + \left( \left( E_a P_\alpha \right)^e \bl \left( P_\beta P_\gamma \right)^f \right) \cup f \cdot e \right) \\
                                                  & = 0
    \end{align*}
    
    \begin{align*}
        (\ref{eq:III-IV-II}) + (\ref{eq:IV-III-II}) & = \sum_{\substack{1 \leq a \leq d_1 \\ 1 \leq \alpha, \beta \leq d_2}} \sum_{\substack{e \in F_{E_a} \wedge P_\alpha \setminus F_{E_a} \\ f = \{ P_\alpha, P_\beta \}}} (-1)^{n(f)} \left( \left( \left( E_a P_\alpha \right)^e \left( P_\alpha P_\beta \right)^f \bl \right) \cup e \cdot f + \left( \left( E_a P_\alpha \right)^e \left( P_\alpha P_\beta \right)^f \bl \right) \cup f \cdot e \right) \\
                                                    & = 0
    \end{align*}
    
    \begin{align*}
        (\ref{eq:IV-IV-I}) & = \sum_{\substack{1 \leq \alpha < \beta < \gamma < \delta \leq d_2}} (-1)^{n(f)} \left( \left( \left( P_\alpha P_\beta \right)^e
            \bl \left( P_\gamma P_\delta \right)^f  \right) \cup e \cdot f + \left( \left( P_\alpha P_\beta \right)^e
        \bl \left( P_\gamma P_\delta \right)^f  \right) \cup f \cdot e  \right)                                                                                                                      \\
                           & + \sum_{\substack{1 \leq \alpha < \gamma < \beta < \delta \leq d_2}} (-1)^{n(f)} \left( \left( \left( P_\alpha P_\beta \right)^e
            \bl \left( P_\gamma P_\delta \right)^f  \right)  \cup e \cdot f + \left( \left( P_\alpha P_\beta \right)^e
        \bl \left( P_\gamma P_\delta \right)^f  \right)  \cup f \cdot e \right)                                                                                                                      \\
                           & + \sum_{\substack{1 \leq \gamma < \alpha < \beta < \delta \leq d_2}} (-1)^{n(f)} \left( \left( \left( P_\gamma P_\delta \right)^f \bl \left( P_\alpha P_\beta \right)^e
            \right)  \cup e \cdot f + \left( \left( P_\gamma P_\delta \right)^f \bl \left( P_\alpha P_\beta \right)^e
        \right)  \cup f \cdot e \right)                                                                                                                                                              \\
                           & = 0
    \end{align*}
    \begin{align*}
        (\ref{eq:IV-IV-II}) + (\ref{eq:IV-IV-III}) & = \sum_{1 \leq \alpha < \beta < \gamma \leq d_2} (-1)^{n(f)} \left( \left( \left( P_\alpha P_\beta \right)^e \left( P_\alpha P_\gamma \right)^f \bl \right) \cup e \cdot f + \left( \left( P_\alpha P_\beta \right)^e \left( P_\alpha P_\gamma \right)^f \bl \right) \cup f \cdot e \right) \\
                                                   & + \sum_{1 \leq \alpha < \beta < \gamma \leq d_2} (-1)^{n(f)} \left( \left( \left( P_\alpha P_\beta \right)^e \left( P_\beta P_\gamma \right)^f \bl \right) \cup e \cdot f + \left( \left( P_\alpha P_\beta \right)^e \left( P_\beta P_\gamma \right)^f \bl \right) \cup f \cdot e \right)   \\
                                                   & + \sum_{1 \leq \alpha < \gamma < \beta \leq d_2} (-1)^{n(f)} \left( \left( \left( P_\alpha P_\gamma \right)^e \left( P_\beta P_\gamma \right)^f \bl \right) \cup e \cdot f + \left( \left( P_\alpha P_\gamma \right)^e \left( P_\beta P_\gamma \right)^f \bl \right) \cup f \cdot e \right) \\
                                                   & = 0
    \end{align*}
    Now, let us show the reverse inclusion in a constructive way. It is obvious for the cases (I) and (II), so we prove only the case (III).
    
    When $n - 2i \geq 0$, the element
    \begin{align*}
        z = \overset{x}{E_1} \bl \dots \bl \overset{x}{E_i} \bl \overset{x}{P_1} \bl \dots \bl \overset{x}{P_{n-2i}} \in C^{i, n-i}(G),
    \end{align*}
    where $\# E_h = 1$, $h = 1, \dots, i$ is itself an element of $\ker \partial^{i, n-i}$ and also an element of
    \begin{equation*}
        \left< \overset{x}{E_1} \bl \dots \bl \overset{x}{E_i} \bl \overset{x}{P_1} \bl \dots \bl \overset{x}{P_{n-2i}} ; \# E_h = 1, h = 1, \dots, i, n - 2i \geq 0 \right>.
    \end{equation*}
    Next, let us take
    \begin{align*}
        z = E_1 \bl \dots \bl \overset{x}{E_{i_t}} \bl \dots \bl E_{d_1} \bl P_1 \bl \dots \bl \overset{x}{P_{k_t'}} \bl \dots \bl P_{d_2} \in C^{i-1, n-i} (G),
    \end{align*}
    where $\# E_h = 1$, $h = 1, \dots, i$, $\begin{cases}
            1 \leq i_1< \dots < i_{t} < \dots < i_p \leq d_1,   \\
            1 \leq k_1< \dots < k_{t'}< \dots < k_{q} \leq d_2, \\
            p + q = j,
        \end{cases}
        \begin{cases}
            d_2 = n - \sum_{h=1}^{d_1} \# V(E_h), \\
            \sum_{h=1}^{d_1} \# E_h = i-1.
        \end{cases}$
    
    Remark that each $E_h$, $h = 1, \dots, d_1$ is not necessarily satisfy $\# E_h = 1$, and even if $\# E_h \geq 3$, it is not necessarily a tree but contains a cycles. Thus, all components are not necessarily colored $x$. Remark also that for fixed $z \in C^{i-1, n-i}(G)$, by adding an edge $e \in E(G) \setminus \bigcup_{h=1}^{d_1} E_h$ we obtain an element of $C^{i, n-i}(G)$, i.e., $z \cup e \in C^{i, n-i}(G)$.
    
    For the given $z \in C^{i-1, n-i}(G)$, consider $z_e = z \cup e$, $e \in E(G) \setminus \bigcup_{h} E_h$, and $Z = \sum_{e \in E(G) \setminus \bigcup_{h} E_h} z_e$. By the former computations we have $\partial^{i, n-i} \left( Z \right) = 0$. The elements $z_e$, $e \in E(G) \setminus \bigcup_h E_h$ hold the property that they share $(i-1)$ edges and the components colored $x$. Notice that the set $E(G) \setminus \bigcup_{h} E_h$ is the set of all edges that can be added to $z \in C^{i-1 ,n-i}(G)$.
    
    For the set $B = \{ e \in E(G); \text{ $e$ connects two components both colored $x$} \}$ the element $Z \in \ker \partial_G^{i, n-i}$ is the sum of $l = \# \left( E(G) \setminus \bigcup_{h=1}^{d_1} E_h \right) - \# B$ enhanced states. We show the $l$ is the minimum number, i.e., for any $S \subsetneq E(G) \setminus \bigcup_h E$ it is impossible to take $Z = \sum_{e \in S} z_e \in C^{i, n-i}(G)$ such that $Z \in \ker \partial_G^{i,n-i}$, where $z_e = z \cup e$, $z \in C^{i-1, n-i}(G)$.
    
    Let assume the minimum number is less than $l$, say $m (< l)$. Then, there exist enhanced states $z_{i_1}, \dots, z_{i_{m}} \in \{ z_1, \dots, z_l \}$ such that 
    \begin{align*}
        \partial_G^{i, n-i}(z_{i_1} + \dots + z_{i_{m}}) = 0.
    \end{align*}
    On the other hand, non-vanishing terms exist in the $\sum_{p=1}^{l-m} \partial_G^{i,n-i}(z_{i_p})$. Actually, for the set $X = \{ e \in E(G); e \in z_{i_p}, p = 1, \dots, m \}$ the boundary 
    \begin{align*}
        \partial_G^{i,n-i} \left( \sum_{1 \leq a \leq d_1} \sum_{e \in (F_{E_a} \setminus E_a) \cap X} \overset{(x)}{E_a^e} \cup e \right) = \sum_{1 \leq a \leq d_1} \sum_{e \in (F_{E_a} \setminus E_a) \cap X} \partial_G^{i, n-i} \left( \overset{(x)}{E_a^e} \cup e \right)
    \end{align*}
    has the term
    \begin{align*}
        \sum_{1 \leq a<b \leq d_1} \left( \sum_{\substack{e \in (F_{E_a} \setminus E_a) \cap X \\ f \in F_{E_b} \setminus E_b}} (-1)^{n(f)} \left( E_a^e \bl E_b^f \right) \cup e \cdot f + \sum_{\substack{e \in F_{E_a} \setminus E_a \\ f \in (F_{E_b} \setminus E_b) \cap X}} (-1)^{n(f)} \left( E_a^e \bl E_b^f \right) \cup f \cdot e \right) \neq 0
    \end{align*}
    for instance. Thus, $l$ is the minimum number.
    
    Next, we show that when we write $z = z_1 + \dots + z_l \in \ker \partial_G^{i, n-i}$, where $l$ is minimum as possible, $z_i$'s satisfy the property that they share $(i-1) $ edges and the components colored $x$. To show the fact, let us take 
    \begin{align*}
        z = E_1 \bl \dots \bl \overset{x}{E_{i_t}} \bl \dots \bl E_{d_1} \bl P_1 \bl \dots \bl \overset{x}{P_{k_t'}} \bl \dots \bl P_{d_2} \in C^{i-2, n-i} (G),
    \end{align*}
    where $E_h \subset E(G)$, $h = 1, \dots, d_1$, $\begin{cases}
            1 \leq i_1< \dots < i_{t} < \dots < i_p \leq d_1,   \\
            1 \leq k_1< \dots < k_{t'}< \dots < k_{q} \leq d_2, \\
            p + q = j,
        \end{cases}
        \begin{cases}
            d_2 = n - \sum_{h=1}^{d_1} \# V(E_h), \\
            \sum_{h=1}^{d_1} \# E_h = i-2.
        \end{cases}$ \\
    For $f_1, f_2, f_3, f_4 \in E(G) \setminus \bigcup_{h} E_h$ with $f_1, f_2 \neq f_3, f_4$ consider
    \begin{align*}
         & z_{f_1 f_2} = z \cup f_1 \cdot f_2 \in C^{i,n-i}(G)  \\
         & z_{f_3 f_4} = z \cup f_3 \cdot f_4 \in C^{i,n-i}(G).
    \end{align*}
    
    Then, any term of $\partial_G^{i, n-i}(z_{f_1 f_2})$ is of the form 
    \begin{equation*}
        (-1)^{n(g)} \left( E_1 \bl \dots \bl \overset{x}{E_{i_{t}}} \bl \dots \bl E_{d_1} \bl P_1 \bl \dots \bl \overset{x}{P_{k_{t'}}} \bl \dots \bl P_{d_2} \right) \cup f_1 \cdot f_2 \cdot g, \quad g \neq f_1, f_2,
    \end{equation*}
    while $\partial_G^{i, n-i}(z_{f_3 f_4})$ has terms of the form 
    \begin{equation*}
        (-1)^{n(g)} \left( E_1 \bl \dots \bl \overset{x}{E_{i_{t}}} \bl \dots \bl E_{d_1} \bl P_1 \bl \dots \bl \overset{x}{P_{k_{t'}}} \bl \dots \bl P_{d_2} \right) \cup f_3 \cdot f_4 \cdot g, \quad g \neq f_3, f_4.
    \end{equation*}
    Obviously these two terms cannot be canceled. This implies that we need to add more enhanced states of the form $z_{f_1 w}$ and $z_{f_3 w'}$ to construct an element of $\ker \partial_G^{i, n-i}$. By this construction, we finally get the sum
    \begin{align*}
        Z = \sum_{w} z_{f_1 w} + \sum_{w'} z_{f_3 w'}
    \end{align*}
    whose number of the enhanced states is more than $l$, which contradicts to the assumption of minimality of $l$. 
    
    By the construction it is obvious that for $z_0 \in C^{i,n-i}$ we cannot take a finite number of enhanced states $z_1, \dots, z_l \in C^{i, n - i}(G)$ such that $z_i$, $i = 0, \dots, l$ share $i - 1$ edges but at least one of the components colored $x$.
    
    Let us next show that for any $z_0 \in \ker \partial_G^{i,n-i}$ we cannot take a finite number of enhanced states $z_h$, $h=1, \dots, l$ such that $Z = z_0 + \dots + z_l \in \ker \partial_G^{i, n-i}$, where $z_h$, $h=1, \dots, l$ share $i$ edges but at least one of the components colored $x$. We show this fact in a constructive way. 
    
    Let us take $z_0 \in C^{i,n-i}(G)$ of the form
    \begin{equation*}
        z_0 = \dots \bl \overset{x}{K_1} \bl \dots \bl \overset{}{K_2} \bl \dots \bl \overset{x}{K_3} \bl \dots \in C^{i, n-i},
    \end{equation*}
    where $K = E$ or $P$.
    
    To cancel the term $\overset{x}{\left( K_1 K_2 \right)^e}$ of $\partial_G^{i, n-i}(z_0)$ we need to add $z_1 \in C^{i,n-i}(G)$ of the form
    \begin{align*}
         & (A): \dots \bl \overset{}{K_1} \bl \dots \bl \overset{x}{K_2} \bl \dots \bl \overset{(x)}{K_3} \bl \dots \text{ or} \\
         & (B): \dots \bl \overset{x}{K_1} \bl \dots \bl \overset{}{K_2} \bl \dots \bl \overset{}{K_3} \bl \dots.
    \end{align*}
    If we take $z_1$ of the form (A), then $\left( \overset{(x)}{K_1 K_3} \right)^e$ appears in $\partial_G^{i,n-i}(z_1)$ while it is 0 in $\partial_G^{i,n-i}(z_0)$. 
    Thus, this choice is 
    inappropriate to construct an element of $\ker \partial_G^{i, n-i}$. On the other hand, if we take $z_1$ of the form (B), then $\left( \overset{x}{K_1 K_3} \right)^e$ appears in $\partial_G^{i, n-i}(z_1)$ while it is 0 in $\partial_G^{i, n-i}(z_0)$, and so this choice is also inappropriate to construct an element of $\ker \partial_G^{i, n-i}$. 
    
    Therefore, enhanced states $z_h$, $h=0, \dots, l$ of the sum $z = z_0 + \dots + z_l \in \ker \partial_G^{i,n-i}$ should hold the property that they have $i - 1$ edges and the position of the color $x$ commonly. Thus, 
    \begin{align*}
         & \sum_{e \in E(G) \setminus \bigcup_{h=1}^{d_1} E_h} (-1)^{n(e)} \left( E_1 \bl \dots \bl \overset{x}{E_{i_{t}}} \bl \dots \bl E_{d_1} \bl P_1 \bl \dots \bl \overset{x}{P_{k_{t'}}} \bl \dots \bl P_{d_2} \right) \cup e ; \\
         & E_h \subset E(G), h = 1, \dots, d_1, \begin{cases}
            1 \leq i_1< \dots < i_{t} < \dots < i_p \leq d_1,   \\
            1 \leq k_1< \dots < k_{t'}< \dots < k_{q} \leq d_2, \\
            p + q = j,
        \end{cases}
        \begin{cases}
            d_2 = n - \sum_{h=1}^{d_1} \# V(E_h), \\
            \sum_{h=1}^{d_1} \# E_h = i-1
        \end{cases}
    \end{align*}
    becomes the generator of $\ker \partial_G^{i, n-i}$, and there does not exist any other type of generators.
\end{proof}

By the similar discussion as Proposition \ref{prop:generator} we can give generators of $\ker \partial_G^{i,j}$ when $i + j = n - 1$ as following.
\begin{prop}\label{prop:generator2}
    Let $G = K_{n-1}^m$ be the graph defined in the previous section. The generators of $\ker (\partial_G^{i, n - i - 1} : C^{i, n - i - 1}(G) \to C^{i + 1, n - i - 1}(G))$ are given as follows.
    
    (I) $\ker \partial_G^{0, n - 1} = \{ 0 \}$. \\
    
    (II) For $i = 1$, \begin{align*}
         & \ker \partial_G^{1, n - 2} = \left< \sum_{e \in E(G) \setminus E} \left( \overset{x}{E} \bl \overset{x}{P_1} \bl \dots \bl \overset{\hat{x}}{P_i} \bl \dots \bl \overset{x}{P_{n-2}} \right) \cup e ; E \subset E(G), \# E = 1 \right> \\
         & \cup \left< \sum_{e \in E(G) \setminus E} \left( \overset{}{E} \bl \overset{x}{P_1} \bl \dots \bl \overset{x}{P_{n-2}} \right) \cup e ; E \subset E(G), \# E = 1 \right> 
    \end{align*}
    
    (III) For $i = 2$, if $n = 4$,
    \begin{align*}
        \ker \partial_G^{2, 1} & = \left< \overset{x}{E_1} \bl E_2 - E_1 \bl \overset{x}{E_2} ; E_1, E_2 \subset E(G), \# E_1 = \# E_2 = 1 \right> \cup \left< \sum_{e \in E(G) \setminus E} \left( \overset{x}{E} \bl \overset{}{P_1} \bl P_2 \right) \cup e ; E \subset E(G), \# E = 2 \right> \\
                               & \cup \left< \sum_{e \in E(G) \setminus E} \left( \overset{}{E} \bl \overset{x}{P_1} \bl \overset{}{P_2} \right) \cup e ; E \subset E(G), \# E = 2 \right>,
    \end{align*}
    and if $n > 4$,
    \begin{align*}
        \ker \partial_G^{2, n - 3} & = \left< \sum_{e \in E(G) \setminus E} \left( \overset{x}{E} \bl \overset{x}{P_1} \bl \dots \bl \overset{\hat{x}}{P_i} \bl \dots \bl \overset{x}{P_{n-3}} \right) \cup e ; E \subset E(G), \# E = 2 \right> \\
                                   & \cup \left< \sum_{e \in E(G) \setminus E} \left( \overset{}{E} \bl \overset{x}{P_1} \bl \dots \bl \overset{x}{P_i} \bl \dots \bl \overset{x}{P_{n-3}} \right) \cup e ; E \subset E(G), \# E = 2\right>.
    \end{align*} \\
    
    (IV) For $i \geq 3$,
    \begin{align*}
        \ker \partial_G^{i, n - i - 1} & = \left< \sum_{e \in E(G) \setminus \bigcup_{h=1}^{d_1} E_h} (-1)^{n(e)} \left( E_1 \bl \dots \bl \overset{x}{E_{i_{t}}} \bl \dots \bl E_{d_1} \bl P_1 \bl \dots \bl \overset{x}{P_{k_{t'}}} \bl \dots \bl P_{d_2} \right) \cup e ; \right. \\
                                       & E_h \subset E(G), h=1, \dots, d_1, \begin{cases}
            1 \leq i_1< \dots < i_{t} < \dots < i_p \leq d_1,   \\
            1 \leq k_1< \dots < k_{t'}< \dots < k_{q} \leq d_2, \\
            p + q = j,
        \end{cases}
        \left. \begin{cases}
            d_2 = n - \sum_{h=1}^{d_1} \# V(E_h), \\
            \sum_{h=1}^{d_1} \# E_h = i-1
        \end{cases}
        \right>
    \end{align*}
    
\end{prop}
\begin{proof}
    The proof of the case (I) is obvious, and the cases (II), (III) when $n > 4$, (IV) are the similar as the proof of the case (III) in Proposition \ref{prop:generator}, so we only give the proof for the case (III) when $n = 4$. By just simple computations, we have 
    \begin{align*}
         & \partial_G^{2, 1} \left( \overset{x}{E_1} \bl E_2 - E_1 \bl \overset{x}{E_2} \right) = 0                                                          \\
         & \partial_G^{2, 1} \left( \sum_{e \in E(G) \setminus E} \left( \overset{x}{E} \bl \overset{}{P_1} \bl P_2 \right) \cup e \right) = 0               \\
         & \partial_G^{2, 1} \left( \sum_{e \in E(G) \setminus E} \left( \overset{}{E} \bl \overset{x}{P_1} \bl \overset{}{P_2} \right) \cup e \right) = 0. 
    \end{align*}
    Let us consider the reverse inclusion. There are three types of enhanced states of $C^{2, 1}(G)$: (A) $\overset{x}{E_1} \bl E_2$, where $E_1, E_2 \subset E(G)$, $\# E_1 = \# E_2 = 1$; (B) $\left( \overset{x}{E} \bl \overset{}{P_1} \bl P_2 \right) \cup e$, where $E \subset E(G)$, $\# E = 2$; (C) $\left( \overset{}{E} \bl \overset{x}{P_1} \bl \overset{}{P_2} \right) \cup e$, where $E \subset E(G)$, $\# E = 2$. 
    
    Consider the case (A). The boundary would be of the form
    \begin{align*}
        \partial_G^{2, 1}(\overset{x}{E_1} \bl E_2) = \sum_{e \in E(G) \setminus E_1 \cup E_2} (-1)^{n(e)} \left( \overset{x}{E_1 E_2} \right)^e \cup e.
    \end{align*}
    On the other hand, the boundary of $\overset{}{E_1} \bl \overset{x}{E_2} \in C^{2,1}(G)$ is also the same one, so to construct the element of $\ker \partial_G^{2,1}$, it suffices to consider
    \begin{align*}
        \overset{x}{E_1} \bl E_2 - E_1 \bl \overset{x}{E_2}
    \end{align*}
    which becomes a generator of $\ker \partial_G^{i, n-i-1}$. For the cases (B) and (C) follow from the similar proof the case (A).
\end{proof}
\section{A splitting property of the chromatic homology for the complete graph}\label{sec:complete}
Using the description of cocycles given in the previous section, we can prove the following proposition.
\begin{prop}\label{prop:snake}
    Let $2 \leq m \leq n-1$, $G = K_{n-1}^m$ be the graph defined above and $e = \{ m, n \} \in E(K_{n-1}^m)$. Then, for any $i,j$ such that $i + j = n - 1, n$, the connecting homomorphism $\gamma^{i,j}$ of the following diagram is a 0-map: \\
    \begin{center}
        \begin{tikzpicture}[]
            \matrix[matrix of math nodes,column sep={60pt,between origins},row
            sep={60pt,between origins},nodes={asymmetrical rectangle}] (s)
            {
            &|[name=05]| 0 & |[name=06]| 0 & |[name=07]| 0 \\
            |[name=00-1]| 0 &|[name=ka]| \ker \partial_{G/e}^{i-1,j} &|[name=kb]| \ker \partial_G^{i,j} &|[name=kc]| \ker \partial_{G-e}^{i,j} \\
            |[name=03]| 0&|[name=A]| C^{i-1,j}(G/e) &|[name=B]| C^{i,j}(G) &|[name=C]| C^{i,j}(G-e) &|[name=01]| 0 \\
            |[name=02]| 0 &|[name=A']| C^{i,j}(G/e) &|[name=B']| C^{i+1,j}(G) &|[name=C']| C^{i+1,j}(G-e) & |[name=04]| 0 \\
            &|[name=ca]| \coker \partial_{G/e}^{i-1,j} &|[name=cb]| \coker \partial_G^{i,j} &|[name=cc]| \coker \partial_{G-e}^{i,j} & |[name=00-2]| 0 \\
            &|[name=08]| 0 & |[name=09]| 0 & |[name=10]| 0 \\
            };
            \draw[->]
            (00-1) edge (ka)
            (ka) edge (A)
            (kb) edge (B)
            (kc) edge (C)
            (05) edge (ka)
            (06) edge (kb)
            (07) edge (kc)
            (A) edge node[auto] {$\alpha^{i-1,j}$} (B)
            (B) edge node[auto] {$\beta^{i,j}$} (C)
            (C) edge (01)
            (A) edge node[auto] {\(\partial_{G/e}^{i-1,j}\)} (A')
            (B) edge node[auto] {\(\partial_G^{i,j}\)} (B')
            (C) edge node[auto] {\(\partial_{G-e}^{i,j}\)} (C')
            (02) edge (A')
            (03) edge (A)
            (A') edge node[auto] {$\alpha^{i,j}$} (B')
            (B') edge node[auto] {$\beta^{i+1,j}$} (C')
            (A') edge (ca)
            (B') edge (cb)
            (C') edge (cc)
            (ca) edge (08)
            (cb) edge (09)
            (cc) edge (10)
            (C') edge (04)
            (cc) edge (00-2)
            ;
            \draw[->,gray] (ka) edge (kb)
            (kb) edge (kc)
            (ca) edge (cb)
            (cb) edge (cc)
            ;
            \draw[->,gray,rounded corners] (kc) -| node[auto,text=black,pos=.7]
                {\(\gamma^{i,j}\)} ($(01.east)+(.5,0)$) |- ($(B)!.35!(B')$) -|
            ($(02.west)+(-.5,0)$) |- (ca);
        \end{tikzpicture}
    \end{center}
\end{prop}
\begin{proof}
    We only give a proof for $i,j$ with $i + j = n$, since the cases $i,j$ with $i + j = n - 1$ follow similarly.
    Let us consider for $i,j$ with $i + j = n$. We can easily check that the elements 
    \begin{align*}
         & \overset{x}{P_1} \bl \dots \bl \overset{x}{P_n} \in \ker \partial_{G-e}^{0, n}                                                                                                                                                        \\
         & \overset{x}{E} \bl \overset{x}{P_1} \bl \dots \bl \overset{x}{P_{n-2}} \in \ker \partial_{G-e}^{1,n-1}, \text{ where $E \subset E(G - e)$, $\# E = 1$}                                                                                \\
         & \overset{x}{E_1} \bl \dots \bl \overset{x}{E_{i}} \bl \overset{x}{P_1} \bl \dots \bl \overset{x}{P_{n - 2i}} \in \ker \partial_{G-e}^{i, n-i}, \text{ where $E_h \subset E(G - e)$, $\# E_h = 1$, $h = 1, \dots, i$, $n - 2i \geq 0$}
    \end{align*}
    are mapped to 0 by $\gamma^{i, n-i}$ respectively. This is because of the following reason. Since there are only edges connecting two components colored $x$ when we take their boundaries $\partial_G^{i, n-i}$, these elements can also be regarded as elements in $\ker \partial_G^{i,n-i}(G)$. Thus, by diagram chasing these elements are mapped to 0 by $\gamma^{i, n-i}$.
    
    For $i \geq 2$ take 
    \begin{align*}
        z = \sum_{f \in E(G - e) \setminus \bigcup_{h=1}^{d_1} E_h} \left( E_1 \bl \dots \bl \overset{x}{E_{i_{t}}} \bl \dots \bl E_{d_1} \bl P_1 \bl \dots \bl \overset{x}{P_{k_{t'}}} \bl \dots \bl P_{d_2} \right) \cup f \in \ker \partial_{G-e}^{i,n-i},
    \end{align*}
    where $E_h \subset E(G)$, $h = 1, \dots, d_1$, $\begin{cases}
            1 \leq i_1< \dots < i_{t} < \dots < i_p \leq d_1,   \\
            1 \leq k_1< \dots < k_{t'}< \dots < k_{q} \leq d_2, \\
            p + q = j,
        \end{cases}
        \begin{cases}
            d_2 = n - \sum_{h=1}^{d_1} \# V(E_h), \\
            \sum_{h=1}^{d_1} \# E_h = i-1.
        \end{cases}$
    
    There are four cases for the deleted edge $e$ to be added to $z \in C^{i,n-i}(G)$ : (i) $\overset{(x)}{E_a} \sim_e \overset{(x)}{E_a}$; (ii) $\overset{(x)}{E_a} \sim_e \overset{(x)}{E_b}$; (iii) $\overset{(x)}{E_a} \sim_e \overset{(x)}{n}$; (iv) $\overset{(x)}{P_\alpha} \sim_e \overset{(x)}{n}$. Here, we only consider the case (i), since other cases follow by similarly.
    
    By taking
    \begin{align*}
        w = (-1)^{n(e)} \left( E_1 \bl \dots \bl \overset{(x)}{\left( E_{a} \right)^e} \bl \dots \bl E_{d_1} \bl P_1 \bl \dots \bl \overset{x}{P_{k_{t'}}} \bl \dots \bl P_{d_2} \right) \cup e \in C^{i,n-i}(G),
    \end{align*}
    we obtain $\beta^{i,n-i}(z + w) = z$. Moreover, we have $z + w \in \ker \partial_G^{i,n-i}$, since 
    \begin{align*}
        z + w = & \sum_{f \in E(G) \setminus \bigcup_{h=1}^{d_1} E_h \cup \{ e \}} (-1)^{n(f)} \left( E_1 \bl \dots \bl \overset{x}{E_{i_{t}}} \bl \dots \bl E_{d_1} \bl P_1 \bl \dots \bl \overset{x}{P_{k_{t'}}} \bl \dots \bl P_{d_2} \right) \cup f \\
                & + (-1)^{n(e)} \left( E_1 \bl \dots \bl \overset{(x)}{\left( E_{a} \right)^e} \bl \dots \bl E_{d_1} \bl P_1 \bl \dots \bl \overset{x}{P_{k_{t'}}} \bl \dots \bl P_{d_2} \right) \cup e                                                 \\  
                & = \sum_{f \in E(G) \setminus \bigcup_{h=1}^{d_1} E_h} (-1)^{n(f)} \left( E_1 \bl \dots \bl \overset{x}{E_{i_{t}}} \bl \dots \bl E_{d_1} \bl P_1 \bl \dots \bl \overset{x}{P_{k_{t'}}} \bl \dots \bl P_{d_2} \right) \cup f,
    \end{align*}
    and thus by Proposition \ref{prop:generator} we can see that $z + w \in \ker \partial_G^{i,n-i}$. Thus, by diagram chasing, the element $z \in \ker \partial_{G-e}^{i,n-i}$ is mapped to 0 by $\gamma^{i,n-i}$.
\end{proof}
By Proposition \ref{prop:snake} we have short exact sequences
\begin{equation}\label{eq:seq}
    0 \to H^{i,j}(G/e) \to H^{i+1,j}(G) \to H^{i+1,j}(G-e) \to 0
\end{equation}
for all $i,j$ with $i + j = n-1, n$. 

Notice that $H^{i,n-i-1}(G)$ are torsion free for any connected graph $G$ and for all $i$ by Lemma \ref{lem:quote}, the sequence (\ref{eq:seq}) always split for all $i, j$ with $i + j = n-1$.

Next, let us prove the sequence (\ref{eq:seq}) split for all $i, j$ with $i + j = n$. Remark that by Lemma \ref{lem:split_i2} the sequence split for all $i \geq 2$ and that we have $H^{0,n}(G) \simeq H^{0,n}(G \setminus e)$ in general. Thus, it suffices to consider the case $i = 1$.

Let us define a map $\varphi^{1,n-1}: C^{1,n-1}(G) \to C^{0,n-1}(G/e)$ by $S = (s,c) \mapsto (s/e, c_e)$, where $s/e$ is the subset of $E(G/e)$ obtained by contracting $e$ from $n$, and $c_e$ is properly defined in Subsection \ref{subsec:ch_hlg}. This induces a well-defined section
\begin{equation*}
    \widetilde{\varphi}^{1,n-1}: H^{1,n-1}(G) \to H^{0,n-1}(G/e).
\end{equation*}
Actually, it is obvious that $\varphi^{1,n-1} (\ker \partial_G^{1,n-1}) \subset \ker \partial_{G/e}^{0, n-1}$. The fact $\varphi^{1,n-1} (\ima \partial_G^{0,n-1}) = 0$ can be seen as follows.

Since the generators of $C^{0, n-1}(G)$ are of the form $\overset{x}{1} \bl \dots \bl \overset{\hat{x}}{i} \bl \dots \bl \overset{x}{n}$, $i = 1, \dots, n$, $\ima \partial_G^{0,n-1}$ consists of the elements of the form
\begin{align*}
    \sum_{\substack{e = \{ i, j \} \\ j \neq i}} \overset{x}{\left\{ i, j \right\}^e} \bl \overset{x}{P_1} \bl \dots \bl \overset{x}{P_{n-1}} \cup e,
\end{align*}
where for each $i = 1, \dots, n$.

Consider the generators of $\ima \partial_G^{0,n-1}$ of the form
\begin{align*}
     & w = \sum_{\substack{e = \{ m, j \}  \\ j \neq m}} \overset{x}{\left\{ m, j \right\}^e} \bl \overset{x}{P_1} \bl \dots \bl \overset{x}{P_{n-1}} \cup e, \text{ and} \\
     & w' = \sum_{\substack{e = \{ j, n \} \\ j \neq n}} \overset{x}{\left\{ j, n \right\}^e} \bl \overset{x}{P_1} \bl \dots \bl \overset{x}{P_{n-1}} \cup e.
\end{align*}
By replacing generators $w$ by $w - w'$, we can assume that two vertices $m, n$ of the generator are distinct components and both colored $x$. Thus, the map $\varphi^{1,n-1}$ sends all elements in $\ima \partial_G^{0, n-1}$ to $0 \in C^{0, n-1}(G/e)$.

Thus, we can see that the map $\varphi^{1,n-1}$ induces a well-defined section
\begin{equation*}
    \widetilde{\varphi}^{1,n-1}: H^{1,n-1}(G) \to H^{0,n-1}(G/e),
\end{equation*}
and thus the sequences also split when $(i, j) = (1, n - 1)$. Hence, we have the following main theorem of the present paper.

\begin{thm}\label{thm:main_split}
    Let $G = K_{n-1}^m$ be the graph defined above and $e = \{ m, n \}$. Then, we have the following split exact sequence
    \begin{equation}
        0 \to H^{i,j}(G/e) \to H^{i+1, j} (G) \to H^{i+1,j} (G - e) \to 0
    \end{equation}
    for all $i, j$ such that $i + j = n-1, n$. 
    
    If we sum over $j$, we have the split exact sequence
    \begin{equation}
        0 \to H^{i}(G/e) \to H^{i+1} (G) \to H^{i+1} (G - e) \to 0
    \end{equation}
    for all $i$.
\end{thm}

Since $K_n/e = K_{n-1}$, we have the following proposition.
\begin{prop}\label{prop:main}
    For $1 \leq j \leq n - 2$, we have the following split exact sequence
    \begin{align*}
        0 \to H^{i,j}(K_{n-1}) \to H^{i+1,j}(K_{n-1}^m) \to H^{i+1,j}(K_{n-1}^{m+1}) \to 0
    \end{align*}
    for all $i,j$ such that $i + j = n-1, n$. \\
    By summing up by $j$ we have the split exact sequence
    \begin{align*}
        0 \to H^{i}(K_{n-1}) \to H^{i+1}(K_{n-1}^m) \to H^{i+1}(K_{n-1}^{m+1}) \to 0
    \end{align*}
    for all $i$.
\end{prop}
Thus, as a corollary of Proposition \ref{prop:main} we obtain a recursive description of the chromatic homology of the complete graph.
\begin{thm}[Conjecture 6.8 \cite{HY}]\label{thm:main}
    For $n \geq 4$ the chromatic homology groups of a complete graph $K_n$ are given recursively as
    \begin{equation}
        H^i (K_n) =
        \begin{cases}
            \ZZ \{ n \}                                                  & i = 0             \\
            H^{i-1} (K_{n-1})^{\oplus (n-2)} \oplus H^i(K_{n-1}) \{ 1 \} & 1 \leq i \leq n-2 \\
            0                                                            & i \geq n-1.
        \end{cases}
    \end{equation}
\end{thm}
\begin{proof}
    To begin with, let us consider $H^0(K_n)$. By Proposition \ref{prop:generator} we can see that $H^0(K_n) \simeq \ker \left( \partial_{K_n}^{0}: C^0(K_n) \to C^1(K_n) \right)$ is generated by only $\overset{x}{P_1} \bl \dots \bl \overset{x}{P_j} \bl \dots \bl \overset{x}{P_n}$. Thus, we have $H^0 (K_n) =  \ZZ \{ n \}$. 
    
    The case $i \geq n-1$ follows by induction on $n$. Though it has been proved in \cite{H}, let us review the proof here. Let assume $H^i(K_{n-1}) = 0$ when $i \geq n-2$. Then, we show $H^i(K_{n}) = 0$ when $i \geq n-1$. \\
    By combing the assumption with the long exact sequence given in Theorem \ref{thm:exct_seq}, we have
    \begin{align*}
        0 \to H^i(K_{n-1}^{m-1}) \to H^i(K_{n-1}^m) \to 0
    \end{align*}
    for $i \geq n-1$, $1 \leq m \leq n-1$. 
    
    Thus, we have
    \begin{align*}
        H^i(K_n) = H^i(K_{n-1}^{n-1}) \simeq H^i(K_{n-1}^{n-2}) \simeq \dots \simeq H^i(K_{n-1}^{1}) \simeq H^i(K_{n-1})\{ 1 \} = 0 \quad \text{for $i \geq n-1$}.
    \end{align*}
    Finally, let us show the case $1 \leq i \leq n-2$. By Proposition \ref{prop:main} we can write $H^i(K_n)$ as follows:
    \begin{align*}
        H^i(K_n) = H^i (K_{n-1}^{n-1}) & = H^{i-1} (K_{n-1}) \oplus H^i (K_{n-1}^{n-2}) \\ 
        H^i (K_{n-1}^{n-2})            & = H^{i-1} (K_{n-1}) \oplus H^i (K_{n-1}^{n-3}) \\
        H^i (K_{n-1}^{n-3})            & = H^{i-1} (K_{n-1}) \oplus H^i (K_{n-1}^{n-4}) \\
        \cdots\cdots                   & \cdots\cdots\cdots\cdots\cdots\cdots\cdots     \\
        H^i (K_{n-1}^2)                & = H^{i-1} (K_{n-1}) \oplus H^i (K_{n-1}^1).
    \end{align*}
    Remark that since $K_{n-1}^{1}$ is a graph obtained by adding a pendant edge, which is an edge of $G$ such that one of the end vertices has no edges, to the complete graph $K_{n-1}$, we have $H^i (K_{n-1}^1) \simeq H^i (K_{n-1})\{ 1 \}$, and thus we obtain
    \begin{align*}
        H^{i-1} (K_{n-1})^{\oplus (n-2)} \oplus H^i(K_{n-1}) \{ 1 \},
    \end{align*}
    which completes the proof.
\end{proof}

\begin{rem}
    Theorem \ref{thm:main} also gives the characteristic homology, introduced in \cite{DL}, of the braid arrangement, which would be the first result for the explicit calculation of the homology.
\end{rem}
\begin{rem}
    It would be the further study to determine which graph $G$ and edge $e$ gives the splitting property of the chromatic homology for all $i$. It would be also interesting to generalize the Theorem \ref{thm:main_split} with the general algebra $\mathbb{Z}[x]/(x^m)$, i.e., the chromatic homology of graphs with colors in $\mathbb{Z}[x]/(x^m)$. Such a generalization might lead to further developments in the study of chromatic homology, including insights into the conjectures given by Pabiniak et al. \cite{PPS}.
\end{rem}

\textit{Acknowledgements:} The author would like to thank Toshiyuki Akita and Takuya Saito for  useful discussion, reading the draft manuscript and giving useful suggestions.


\end{document}